\documentclass[a4paper,10pt]{article}

\usepackage{authblk}
\usepackage[USenglish]{babel}
\usepackage{geometry}
\usepackage{lmodern}
\usepackage{amsmath,amsfonts,amsthm}
\usepackage{mathrsfs}
\usepackage[inline, shortlabels]{enumitem}
\usepackage{float}
\usepackage{graphicx}
\usepackage[font=small,labelfont=bf]{caption}
\usepackage[dvipsnames]{xcolor}
\usepackage[nottoc]{tocbibind}
\usepackage[hyphens]{url}
\usepackage[hidelinks]{hyperref}
\usepackage{cleveref}

\newcommand\MyFigScale{0.8205128205}

\theoremstyle{plain}
\newtheorem{theorem}{Theorem}
\newtheorem{corollary}[theorem]{Corollary}

\theoremstyle{definition}
\newtheorem{claim}[theorem]{Claim}

\newtheorem{proposition}[theorem]{Proposition}
\newtheorem{question}[theorem]{Question}
\newtheorem{definition}[theorem]{Definition}
\newtheorem{example}[theorem]{Example}
\newtheorem{remark}[theorem]{Remark}

\newcommand{\altmanifold}{\mathcal{N}}
\newcommand{\altsurface}{\mathcal{F}}
\newcommand{\altaltsurface}{\mathcal{R}}
\newcommand{\alttree}{T}
\newcommand{\compbody}{\mathcal{C}}
\newcommand{\compbodyone}{\mathcal{N}}
\newcommand{\compbodytwo}{\mathcal{K}}
\newcommand{\dual}{\Gamma}
\newcommand{\fork}{F}
\newcommand{\forkcomp}{\boldsymbol{F}}
\newcommand{\hbody}{\mathcal{H}}
\newcommand{\inj}{\operatorname{inj}}
\newcommand{\manifold}{\mathcal{M}}
\newcommand{\sd}{\operatorname{sd}}
\newcommand{\surface}{\mathcal{S}}
\newcommand{\torus}{\mathbb{T}}
\newcommand{\tri}{\mathcal{T}}
\newcommand{\twodisk}{\mathbb{D}}

\newcommand{\heeg}[1]{\mathfrak{g}\left(#1\right)}
\newcommand{\nbh}[1]{N(#1)}
\newcommand{\nsphere}[1]{\mathbb{S}^{#1}}
\newcommand{\ov}[1]{\overline{#1\vphantom{\scalebox{1.2}{$#1$}}}}
\newcommand{\pw}[1]{\operatorname{pw}(#1)}
\newcommand{\tw}[1]{\operatorname{tw}(#1)}
\newcommand{\vol}[1]{\operatorname{vol}(#1)}

\title{On the pathwidth of hyperbolic 3-manifolds\thanks{This paper is based on previously unpublished parts of the author's PhD thesis \cite{huszar2020combinatorial}.}}

\author{Krist\'of Husz\'ar\thanks{Supported by the French government through the 3IA C\^ote d'Azur Investments in the Future project managed by the National Research Agency (ANR) under the reference number ANR-19-P3IA-0002.} \space\href{https://orcid.org/0000-0002-5445-5057}{\includegraphics[height=.9em]{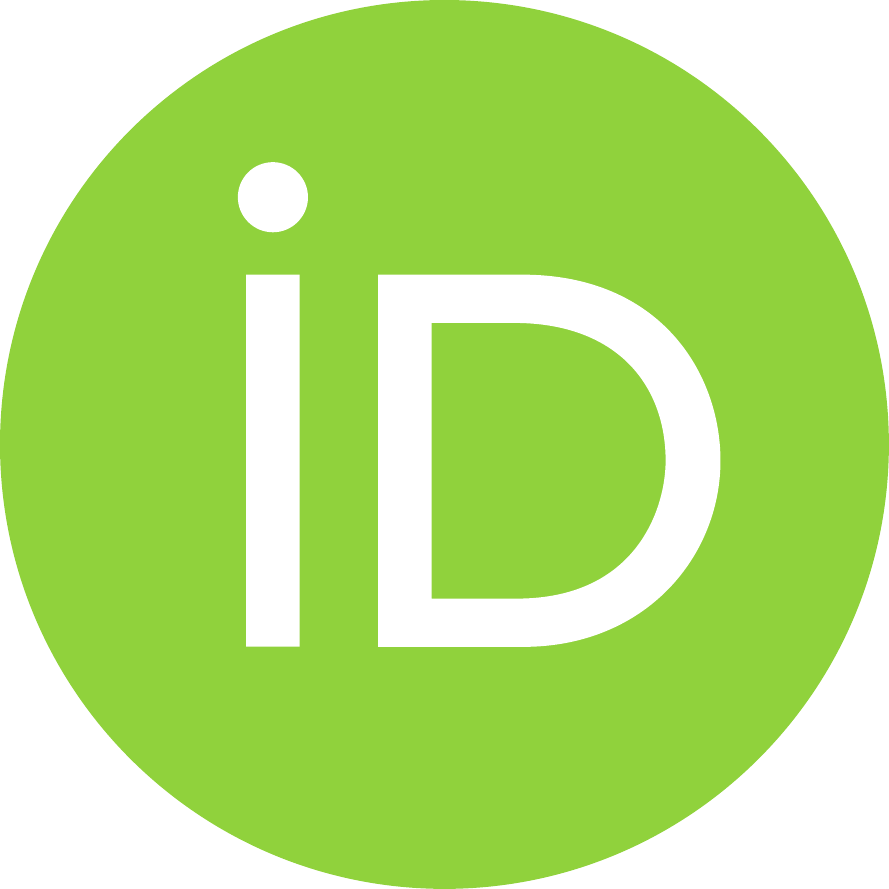}} \\ Inria Sophia Antipolis - M\'editerran\'ee \\ 2004 route des Lucioles, 06902 Sophia Antipolis, France}

\date{May 24, 2021}

\begin{document}

\maketitle

\begin{abstract}
According to Mostow's celebrated rigidity theorem, the geometry of closed hyperbolic 3-manifolds is already determined by their topology. In particular, the volume of such manifolds is a topological invariant and, as such, has been investigated for half a century.

Motivated by the algorithmic study of 3-manifolds, Maria and Purcell have recently shown that every closed hyperbolic 3-manifold $\manifold$ with volume $\vol{\manifold}$ admits a triangulation with dual graph of treewidth at most $C\cdot\vol{\manifold}$, for some universal constant $C$.

Here we improve on this result by showing that the volume provides a linear upper bound even on the pathwidth of the dual graph of some triangulation, which can potentially be much larger than the treewidth.
Our proof relies on a synthesis of tools from 3-manifold theory: generalized Heegaard splittings, amalgamations, and the thick-thin decomposition of hyperbolic 3-manifolds. We provide an illustrated exposition of this toolbox and also discuss the algorithmic consequences of the result.
\end{abstract}

\section{Introduction}
\label{sec:intro}

Algorithms in computational 3-manifold topology typically take a triangulation as input and return topological information about the underlying manifold. The difficulty of extracting the desired information, however, might greatly depend on the choice of the input triangulation. In recent years, several computationally hard problems about triangulated 3-manifolds were shown to admit algorithmic solutions that are {\em fixed-parameter tractable} (FPT) in the {\em treewidth}\footnote{The treewidth is a structural graph parameter measuring the ``tree-likeness'' of a graph, cf.\ Section \ref{sec:comb}.} of the dual graph of the input triangulation \cite{burton2017courcelle, burton2016parameterized, burton2018algorithms, pettersson2014fixed, burton2013complexity}.\footnote{For related work on FPT-algorithms in knot theory, see \cite{burton2018homfly} and \cite{maria2019parametrized} and the references therein.} These algorithms still require exponential time to terminate in the worst case. However, for triangulations with dual graph of bounded treewidth they run in polynomial time.\footnote{The running times are measured in terms of the number of tetrahedra in the input triangulation.}\textsuperscript{,}\footnote{Some of these algorithms \cite{burton2018algorithms, burton2013complexity} have been implemented in the topology software {\em Regina} \cite{burton2013regina, Regina}.}

In the light of these algorithms, it is compelling to consider the {\em treewidth $\tw{\manifold}$ of a compact $3$-manifold $\manifold$}, defined as the smallest treewidth of the dual graph of any triangulation thereof. Over the last few years, the quantitative relationship between the treewidth and other properties of 3-manifolds has been studied in various settings.\footnote{See \cite{mesmay2019treewidth} for related work in knot theory concerning a different notion of treewidth for knot diagrams.} The author together with Spreer showed, for instance, that the {\em Heegaard genus} always gives an upper bound on the treewidth (even on the {\em pathwidth}) \cite{huszar2019manifold}, and together with Wagner they established that, for certain families of 3-manifolds the treewidth can be arbitrary large \cite{huszar2019treewidth}.\footnote{For the precise statements of these results cf.\ the inequalities \eqref{eq:heeg-tw} and \eqref{eq:heeg-pw} in Section \ref{sec:comb}. For further results and a detailed discussion we refer to the author's PhD thesis \cite{huszar2020combinatorial}.}

Recently, Maria and Purcell have shown that, in the realm of {\em hyperbolic} 3-manifolds another important invariant, the {\em volume}, yields an upper bound on the treewidth \cite{maria2019treewidth}. They proved the existence of a universal constant $C > 0$, such that, for every closed hyperbolic 3-manifold $\manifold$ with treewidth $\tw{\manifold}$ and volume $\vol{\manifold}$ the following inequality holds:
\begin{align}
	\tw{\manifold} \leq C \cdot \vol{\manifold}. 
	\label{eq:tw-vol}
\end{align}

In this article we improve upon \eqref{eq:tw-vol} by showing that the volume provides a linear upper bound even on the {\em pathwidth} of a hyperbolic 3-manifold---a quantity closely related to, but potentially much larger than the treewidth. More precisely, we prove the following theorem.

\begin{theorem}
\label{thm:pw-vol}
There exists a universal constant $C' > 0$ such that, for any closed, orientable and hyperbolic $3$-manifold $\manifold$ with pathwidth $\pw{\manifold}$ and volume $\vol{\manifold}$, we have
\begin{alignat}{1}
	\pw{\manifold} \leq C' \cdot \vol{\manifold}.
	\label{eq:pw-vol}
\end{alignat}
\end{theorem}

\paragraph{Outline of the proof.}
Our roadmap to establish Theorem \ref{thm:pw-vol} is similar to that in \cite{maria2019treewidth}. In particular, our construction of a triangulation of $\manifold$ with dual graph of pathwidth bounded in terms of $\vol{\manifold}$ also starts with a {\em thick-thin decomposition} $\mathscr{D}$ of $\manifold$. The two proofs, however, diverge at this point. Maria and Purcell proceed by triangulating the thick part of $\mathscr{D}$ using the work of J{\o}rgensen--Thurston {\cite[\S 5.11]{thurston2002geometry}} and Kobayashi--Rieck \cite{kobayashi2011linear}. This partial triangulation is then simplified \cite{burton2014crushing, jaco2003efficient} and completed into the desired triangulation of $\manifold$.

The novelty in our work is, that we proceed by first turning the decomposition $\mathscr{D}$ into a {\em generalized Heegaard splitting} of $\manifold$ \cite{scharlemann2016lecture, scharlemann1992thin}, where we rely on the aforementioned results to control the genera of the splitting surfaces. Next, we {\em amalgamate} this generalized Heegaard splitting into a classical one \cite{schultens1993classification}. Finally, we appeal to our earlier work \cite{huszar2019manifold} to turn this Heegaard splitting into a triangulation of $\manifold$ with dual graph of pathwidth $O(\vol{\manifold})$.

\bigskip

The proof of Theorem \ref{thm:pw-vol} provides a template for an algorithm\footnote{We refer to the discussion in \cite[Section 5.1]{maria2019treewidth} for the description of a possible computational model.} to triangulate any closed hyperbolic 3-manifold $\manifold$ in such a way, that the dual graph of the resulting triangulation has pathwidth $O(\vol{\manifold})$. Using such triangulations---that have a dual graph not only of small treewidth, but also pathwidth---as input for FPT-algorithms may significantly reduce their running time. This is because such triangulations lend themselves to {\em nice tree decompositions} (the data structure underlying many algorithms FPT in the treewidth) without {\em join bags} (those parts of a nice tree decomposition that often account for the computational bottleneck, cf.\ \cite{burton2016parameterized}).
The upshot of Theorem \ref{thm:pw-vol} is that, in case of hyperbolic 3-manifolds with bounded volume working with such triangulations is (in theory) always possible.

\paragraph{Structure of the paper.} We start with an illustrated exposition of the various notions from 3-manifold theory we rely on (Section \ref{sec:3mfds}). Then, in Section \ref{sec:comb}, we discuss the treewidth and pathwidth for graphs and 3-manifolds alike. In Section \ref{sec:proof}, we put all the pieces together to prove Theorem \ref{thm:pw-vol}. We conclude with a discussion and some open questions in Section \ref{sec:discussion}.

\paragraph{Acknowledgements.} I am grateful to Uli Wagner and Jonathan Spreer for their guidance and steady support during my PhD, and \'Eric Colin de Verdi\`ere and Herbert Edelsbrunner for their careful review of my thesis. I thank the anonymous referees for their useful suggestions to improve the exposition. I also thank Cl\'ement Maria for many stimulating discussions, and my colleagues at Inria for the warm welcome in Sophia Antipolis in the midst of a pandemic.

\section{A primer on 3-manifolds}
\label{sec:3mfds}

The main objects of study in this paper are 3-dimensional manifolds, or $3$-manifolds for short. As we will also encounter 2-manifolds, also known as {\em surfaces}, we give the general definition. A {\em $d$-dimensional manifold with boundary} is a topological space\footnote{More precisely, we only consider topological spaces which are {\em second countable} and {\em Hausdorff}.} $\manifold$ such that each point $x \in \manifold$ has a neighborhood which looks like (i.e., is homeomorphic to) the Euclidean $d$-space $\mathbb{R}^d$ or the closed upper half-space $\{(x_1,\ldots,x_d) \in \mathbb{R}^d : x_d\geq 0\}$.
The points of $\manifold$ that do not have a neighborhood homeomorphic to $\mathbb{R}^d$ constitute the {\em boundary $\partial \manifold$} of $\manifold$. A compact manifold is said to be {\em closed} if it has an empty boundary.

Two manifolds $\manifold_1$ and $\manifold_2$ are considered equivalent if they are {\em homeomorphic}, i.e., if there exists a continuous bijection $f\colon\manifold_1\rightarrow\manifold_2$ with $f^{-1}$ being continuous as well. Properties of manifolds that are preserved under homeomorphisms are called {\em topological invariants}.
We refer to \cite{schultens2014introduction} for an introduction to 3-manifolds (cf.\ \cite{hempel2004manifolds, jaco1980lectures, saveliev2012lectures, thurston2014three}).

All $3$-manifolds in this paper are assumed to be compact and orientable.

\subsection{Triangulations and handle decompositions}
\label{ssec:trg}

\paragraph{Triangulations.} By a classical result of Moise \cite{moise1952affine} (cf.\ \cite{bing1959alternative}) every compact 3-manifold admits a triangulation. To build a triangulation, take a disjoint union $\widetilde{\Delta} = \Delta_1 \cup\ldots\cup\Delta_n$ of finitely many tetrahedra with $4n$ triangular faces altogether. Let $\Phi = \{\varphi_1,\ldots,\varphi_m\}$ be a set of at most $2n$ {\em face gluings}, each of which identifies a pair of these triangular faces in such a way that vertices are mapped to vertices, edges to edges, and each face is identified with at most one other face, see Figure \ref{fig:tetrahedra}(i). The resulting quotient space $\tri = \widetilde{\Delta} / \Phi$ is called a {\em triangulation}, and the pairs of identified triangular faces are referred to as {\em triangles} of $\tri$. Note that these face gluings might identify several tetrahedral edges (or vertices) of $\widetilde{\Delta}$ resulting in a single {\em edge} (or {\em vertex}) of $\tri$.

To obtain a triangulation $\tri$ that is homeomorphic to a closed $3$-manifold $\manifold$, it is necessary and sufficient that the boundary of a small neighborhood around each vertex is a sphere, and no edge is identified with itself in reverse. If some of the vertices have small neighborhoods with boundaries being disks, then $\tri$ describes a $3$-manifold with boundary. In a computational setting, a 3-manifold is very often presented this way.

In the study of triangulations, their dual graphs play an instrumental role.\footnote{Following a convention adopted by several authors in the field of computational low-dimensional topology, throughout this paper we use the terms {\em edge} and {\em vertex} to refer to an edge or vertex in a $3$-manifold triangulation, whereas the terms {\em arc} and {\em node} denote an edge or vertex in a graph, respectively.}
Given a triangulation $\tri = \widetilde{\Delta} / \Phi$, its {\em dual graph} $\Gamma (\tri) = (V,E)$ is a multigraph\footnote{In a multigraph $G=(V,E)$ the set $E$ of arcs is a multiset, i.e., there might be multiple arcs running between two given nodes. Moreover, an arc itself can also be a multiset in which case it is called a {\em loop}. Next, when talking about graphs, we will always mean multigraphs, unless otherwise stated.} where the nodes in $V$ correspond to the tetrahedra in $\widetilde{\Delta}$, and for each face gluing $\varphi \in \Phi$ identifying two triangular faces of $\Delta_i$ and $\Delta_j$, we add an arc between the corresponding nodes in $V$, cf.\ Figure \ref{fig:tetrahedra}(ii). (Note that $i$ and $j$ could be equal.) By construction, every node of $\Gamma (\tri)$ has maximum degree $\leq 4$. Moreover, when $\tri$ triangulates a closed 3-manifold, then $\Gamma (\tri)$ is 4-regular.

\begin{figure}[ht]
	\centerline{\includegraphics[scale=\MyFigScale]{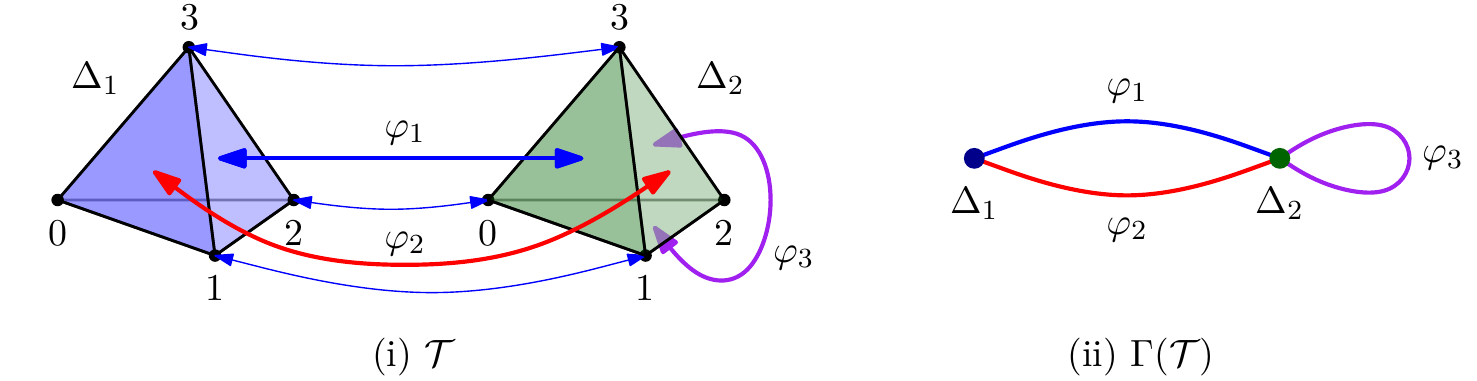}}
	\caption{(i) A triangulation $\tri = \widetilde{\Delta} / \Phi$ with two tetrahedra $\widetilde{\Delta} = \{\Delta_1,\Delta_2\}$ and three face gluing maps $\Phi = \{\varphi_1,\varphi_2,\varphi_3\}.$ $\varphi_1$ is specified to be $\Delta_1 (123) \protect\overset{\displaystyle\varphi_1}{\longleftrightarrow}\Delta_2 (103)$. (ii) The dual graph $\Gamma (\tri)$ of $\tri$.}
	\label{fig:tetrahedra}
\end{figure}

\paragraph{Handle decompositions.} It follows from Morse theory (and also from the existence of triangulations) that every compact $3$-manifold can be built from finitely many solid building blocks called $0$-, $1$-, $2$-, and {\em $3$-handles}. In such a {\em handle decomposition} all handles are homeomorphic to $3$-balls, and are only distinguished in how they are glued together.
To construct a closed $3$-manifold from handles, we may start with a disjoint union of $3$-balls, or {\em $0$-handles}, where further $3$-balls are glued to the boundary of the existing decomposition along pairs of $2$-dimensional disks ({\em $1$-handles}), or along annuli ({\em $2$-handles}). This process is iterated until the boundary consists of a disjoint union of $2$-spheres. These are then eliminated by gluing in one $3$-ball per boundary component, the {\em $3$-handles} of the decomposition. 

\begin{figure}[ht]
	\centerline{\includegraphics[scale=\MyFigScale]{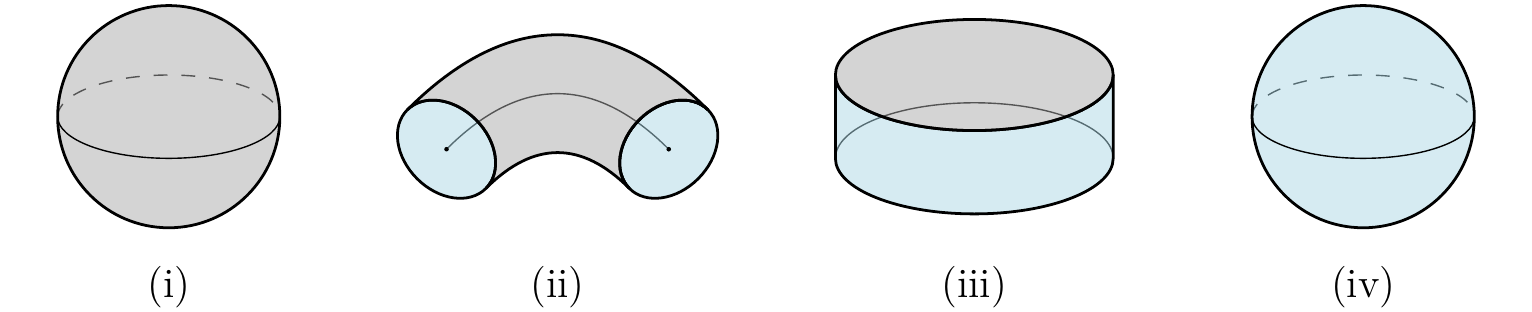}}
	\caption{(i) A 0-handle, (ii) a 1-handle, (iii) a 2-handle, and (iv) a 3-handle. The attaching sites are indicated with light blue. For a 1-handle, this is a disjoint union of two disks, for a 2-handle an annulus, and for a 3-handle the entire 2-sphere boundary.}
	\label{fig:handles}
\end{figure}

\subsection{Handlebodies and compression bodies}
\label{ssec:bodies}

A {\em handlebody} $\hbody$ is a connected $3$-manifold with boundary that is built from (finitely many) $0$-handles and  $1$-handles. It can also be seen as a thickened graph. Up to homeomorphism, a handlebody $\hbody$ is determined by the genus $g(\partial\hbody)$ of its boundary.

Let $\surface$ be a compact, orientable (not necessarily connected) surface. A {\em compression body} is a $3$-manifold $\compbody$ obtained from $\surface \times [0,1]$ by (optionally) attaching some $1$-handles to $\surface \times \{1\}$, and (optionally) filling in some of the $2$-sphere components of $\surface \times \{0\}$ with $3$-balls. $\compbody$ has two sets of boundary components: $\partial_- \compbody = \surface \times \{0\} \setminus \{\text{filled-in 2-sphere components}\}$ and $\partial_+ \compbody = \partial \compbody \setminus \partial_- \compbody$. We call $\partial_{+} \compbody$ the {\em upper boundary}, and $\partial_{-} \compbody$ the {\em lower boundary} of $\compbody$.

Dual to this construction, a compression body $\compbody$ can also be built by starting with a closed, orientable surface $\altsurface$, thickening it to $\altsurface \times [0,1]$, (optionally) attaching some $2$-handles along $\altsurface \times \{0\}$, and (optionally) filling in some of the resulting $2$-spheres with $3$-balls. The upper and lower boundary are given by $\partial_{+} \compbody = \altsurface \times \{1\}$ and $\partial_{-} \compbody = \partial \compbody \setminus \partial_{+} \compbody$. 

Note that every handlebody is also a compression body, where all 2-sphere components are eliminated in the last step.

See Figure \ref{fig:compbody} in Appendix \ref{app:compbody} for an illustration of the primal and dual constructions.

\subsection{Heegaard splittings}
\label{ssec:heegaard}

Introduced in \cite{heegaard1916analysis}, Heegaard splittings have been central to the study of 3-manifolds for over a century.
Given a closed, orientable 3-manifold $\manifold$, a {\em Heegaard splitting} is a decomposition $\manifold = \hbody \cup_\surface \hbody'$ where $\hbody$ and $\hbody'$ are homeomorphic handlebodies with $\hbody \cup \hbody' = \manifold$ and $\hbody \cap \hbody' = \partial\hbody = \partial\hbody' = \surface$ called the {\em splitting surface}. The {\em Heegaard genus} $\heeg{\manifold}$ of $\manifold$ is the smallest genus $g(\surface)$ over all Heegaard splittings of $\manifold.$ See \cite{scharlemann2002heegaard} for a comprehensive survey.

\begin{example}[Heegaard splittings from triangulations, I] Given a triangulation $\tri$ of a closed, orientable 3-manifold $\manifold$, let $\tri^{(1)}$ denote its $1$-skeleton consisting of the vertices and edges of $\tri$. Thickening up $\tri^{(1)}$, i.e., taking its regular neighborhood, in $\manifold$ yields a handlebody $\hbody_1$. The closure $\hbody_2$ of the complement $\manifold \setminus \hbody_1$  is also a handlebody homeomorphic to a regular neighborhood of $\Gamma(\tri)$, and $\manifold = \hbody_1 \cup \hbody_2$ is a Heegaard splitting of $\manifold$.
\label{ex:heegaard}
\end{example}

\paragraph{Heegaard splittings of 3-manifolds with boundary.} Using compression bodies, one can generalize Heegaard splittings to 3-manifolds with nonempty boundary. Let $\manifold$ be a $3$-manifold and $\partial_1\manifold \cup \partial_2\manifold = \partial\manifold$ be an arbitrary partition of its boundary components. There exist compression bodies $\compbody_1$ and $\compbody_2$ with $\compbody_1 \cup \compbody_2 = \manifold$, $\partial_-\compbody_1 = \partial_1\manifold$, $\partial_-\compbody_2 = \partial_2\manifold$, and $\compbody_1 \cap \compbody_2 = \partial_+\compbody_1 = \partial_+\compbody_2$. The decomposition $\manifold = \compbody_1 \cup_\surface \compbody_2$ is called a Heegaard splitting of $\manifold$ {\em compatible with the partition} $\partial_1\manifold \cup \partial_2\manifold$. Its {\em splitting surface} is $\surface = \compbody_1 \cap \compbody_2$. The Heegaard genus $\heeg{\manifold}$ is again the minimum genus $g(\surface)$ over all such decompositions.

See Example \ref{ex:heegaard_boundary} in Appendix \ref{app:heegaard} for an extension of Example \ref{ex:heegaard} to this setting.

\subsection{Generalized Heegaard splittings}
\label{ssec:gen_heegaard}

The notion of a Heegaard splitting, where a 3-manifold is built by gluing two handlebodies together (or two compression bodies, in case of 3-manifolds with boundary), was refined by Scharlemann and Thompson in a seminal paper \cite{scharlemann1992thin}. In a {\em generalized Heegaard splitting} a 3-manifold is constructed from several pairs of compression bodies. This construction arises naturally, e.g., when a 3-manifold is assembled by first attaching only some of the $0$- and $1$-handles before attaching any $2$- and $3$-handles.

Informally, a {\em generalized Heegaard splitting} of a $3$-manifold $\manifold$ is a decomposition
\begin{align}
	\mathscr{D} = \left\{ \manifold_i : i \in I,~\textstyle\bigcup_{i \in I}\manifold_i = \manifold,~\text{and}~\operatorname{int}(\manifold_i) \cap \operatorname{int}(\manifold_j) = \emptyset~\text{for}~i \neq j \right\}
	\label{eq:decomp}
\end{align}
into finitely many $3$-manifolds with pairwise disjoint interiors that intersect along surfaces, together with an ``appropriate'' Heegaard splitting for each $\manifold_i$. We make this now precise.

Given a decomposition $\mathscr{D}$ as above, consider its {\em dual graph},\footnote{Not to be confused with the dual graph of a triangulation.} which is a multigraph $\Gamma(\mathscr{D}) = (I,E)$ with nodes corresponding to the $\manifold_i$ and arcs between $i$ and $j$ to the connected components of $\manifold_i \cap \manifold_j$ (Figure \ref{fig:manifold_decomposition}). Pick an ordering of $I$, i.e., a bijection $\ell\colon I \rightarrow \{1,\ldots,|I|\}$. For any $i \in I$, let $\partial_1 \manifold_i \cup \partial_2 \manifold_i$ be a partition of the connected components of $\partial \manifold_i$ so that $\partial_1 \manifold_i$ (resp.\ $\partial_2 \manifold_i$) contains the components glued to those of any $\manifold_j$ with $\ell(j) < \ell(i)$ (resp.\ $\ell(j) > \ell(i)$). Those components of $\partial \manifold_i$ which contribute to the boundary of $\manifold$ are partitioned among $\partial_1 \manifold_i$ and $\partial_2 \manifold_i$ arbitrarily. For each $i \in I$, choose a Heegaard splitting $\manifold_i = \compbodyone_i \cup_{\surface_i}\compbodytwo_i$ of $\manifold_i$ compatible with the partition $\partial_1 \manifold_i \cup \partial_2 \manifold_i$ of the boundary components (cf.\ Example \ref{ex:heegaard_boundary}). We obtain a {\em generalized Heegaard splitting} of $\manifold$ (Figure \ref{fig:graph_splitting_example}).

\begin{figure}[H]
	\centering
	\includegraphics[scale=\MyFigScale]{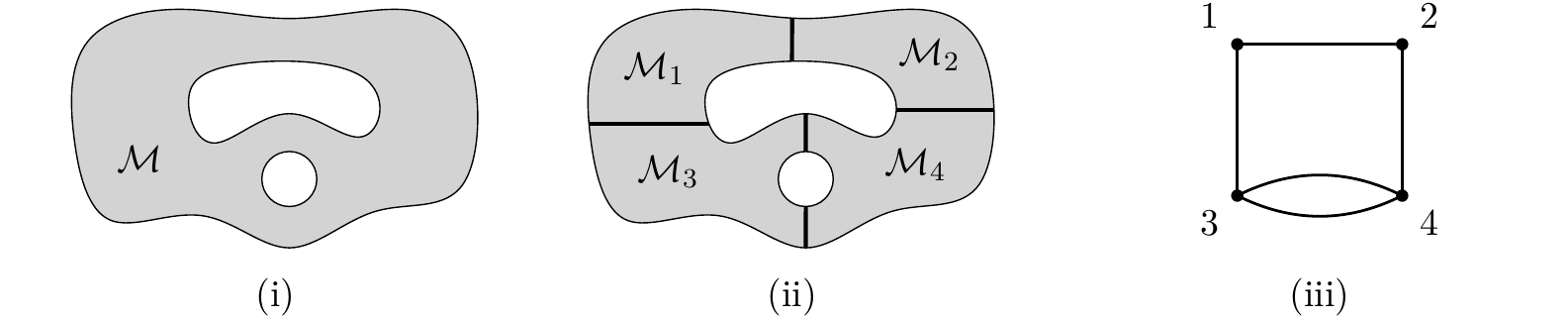}
	\caption{(i) Schematic of a closed 3-manifold $\manifold$ with nontrivial first homology, (ii) a decomposition $\mathscr{D}$ of $\manifold$ into four submanifolds, and (iii) the dual graph $\Gamma(\mathscr{D})$ of $\mathscr{D}$.}
	\label{fig:manifold_decomposition}
\end{figure}

\begin{figure}[H]
	\centering
	\includegraphics[scale=\MyFigScale]{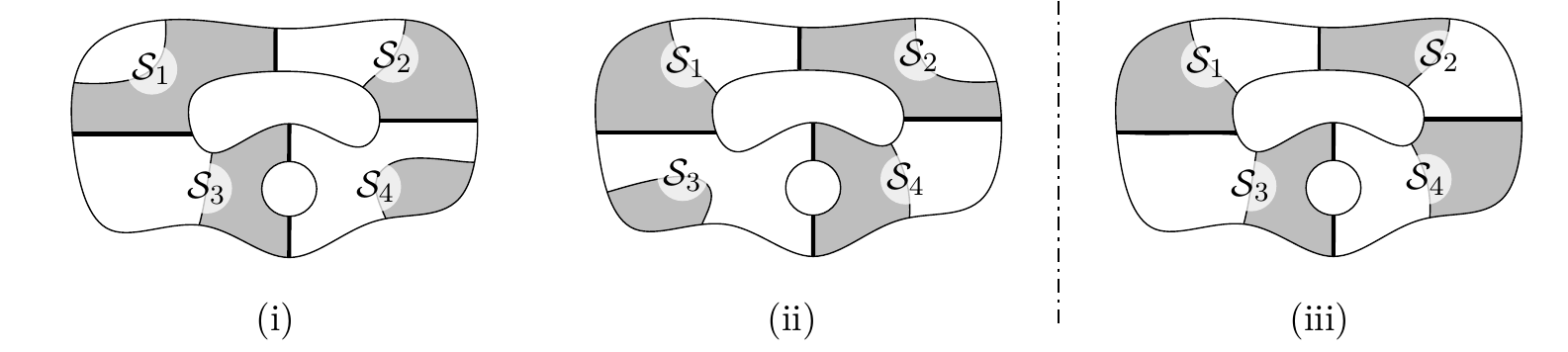}
	\caption{(i)--(ii) Schematics of two generalized Heegaard splittings of $\manifold$ stemming from the decomposition shown on Figure \ref{fig:manifold_decomposition}. These splittings respectively correspond to the orderings $ \ell_1(i)=i$ ($i \in I = \{1,2,3,4\}$), and $\ell_2(1) = 2$, $\ell_2(2) = 4$, $\ell_2(3) = 1$, $\ell_2(4) = 3$. (iii) A non-example.}
	\label{fig:graph_splitting_example}
\end{figure}

When we only need to talk about the constituents of a decomposition or the pieces of a generalized Heegaard splitting of a 3-manfiold $\manifold$ we use the shorthand notation
\begin{align}
\manifold = \bigcup_{i \in I}\manifold_i \quad \text{or} \quad \manifold = \bigcup_{i \in I}(\compbodyone_i \cup_{\surface_i}\compbodytwo_i), \quad \text{where}\quad \compbodyone_i \cup_{\surface_i}\compbodytwo_i = \manifold_i.
	\label{eq:graph}
\end{align}

\paragraph{Fork complexes.} When connectivity properties of the graph $\Gamma(\mathscr{D})$ underlying a given splitting are relevant, it may be more convenient to work with so-called fork complexes. Here we give a brief overview of this language. For more details, see \cite[Chapter 5]{scharlemann2016lecture}.

A {\em fork complex} is essentially a decorated version of $\Gamma(\mathscr{D})$. It is a labeled graph in which the compression bodies of a given decomposition are modeled by {\em fork}s. More precisely, an {\em $n$-fork} is a tree $\fork$ with $n+2$ nodes $V(\fork)=\{g,p,t_1,\ldots,t_n\}$ with $p$ being of degree $n+1$ and all other nodes being leaves. The nodes $g$, $p$, and the $t_i$ are called the $grip$, the $root$, and the $tine$s of $\fork$, respectively (Figure \ref{fig:forkcomp_examples}(i) shows a $0$- and a $3$-fork).
We think of a fork $\fork = \fork_\compbody$ as an abstraction of a compression body $\compbody$, such that the grip of $\fork$ corresponds to $\partial_+\compbody$, whereas the tines correspond to the connected components of $\partial_-\compbody$.

\begin{figure}[ht]
	\centering
	\includegraphics[scale=\MyFigScale]{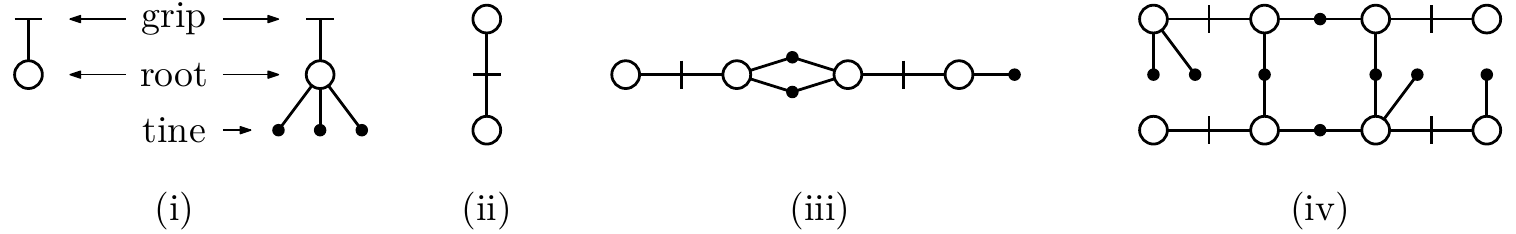}
	\caption{Fork complexes of a Heegaard splitting (ii) and of generalized Heegaard splittings (iii)--(iv). Reproduced from \cite[Figure 1]{huszar2019treewidth}.}
	\label{fig:forkcomp_examples}
\end{figure}

Informally, a {\em fork complex} $\forkcomp$ (representing a given generalized Heegaard splitting of a $3$-manifold $\manifold$) is obtained by taking several forks (corresponding to the compression bodies which constitute $\manifold$), and identifying grips with grips, and tines with tines (following the way the boundaries of these compression bodies are glued together). The set of grips and tines which remain unpaired is denoted by $\partial\forkcomp$ (as they correspond to surfaces which constitute the boundary $\partial\manifold$). See Figure \ref{fig:forkcomp_examples} for illustrations, and \cite[Section 5.1]{scharlemann2016lecture} for further details.

\begin{figure}[ht]
	\centering
	\includegraphics[scale=\MyFigScale]{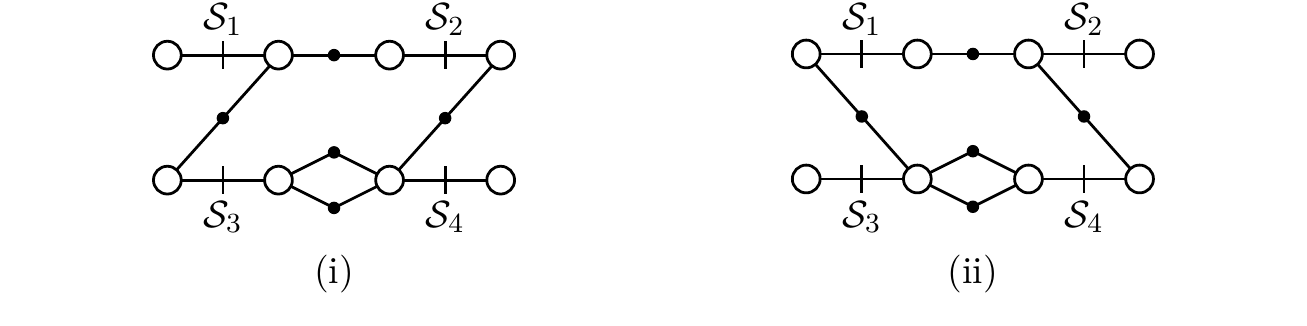}
	\caption{Fork complexes representing the generalized Heegaard splittings shown on Figure \ref{fig:graph_splitting_example}.}
	\label{fig:forkcomp_two_splittings}
\end{figure}

\paragraph{Amalgamations.} Introduced by Schultens in \cite{schultens1993classification}, {\em amalgamation} is a useful procedure that turns a generalized Heegaard splitting into a classical one. There are several excellent references where amalgamations are discussed in detail (cf.\ \cite[Section 2]{bachman2017computing}, \cite[Section 2.3]{derby-talbot2009stabilization}, \cite[Section 5.4]{scharlemann2016lecture}), therefore here we rely on a simple example to illustrate this operation.

Let $\manifold = (\compbodyone_1 \cup_{\surface_1}\compbodytwo_1) \cup_{\altaltsurface} (\compbodyone_2 \cup_{\surface_2}\compbodytwo_2)$ be a generalized Heegaard splitting of $\manifold$, which we would like to amalgamate to form a classical Heegaard splitting $\manifold = \compbodyone \cup_\surface \compbodytwo$, see Figure \ref{fig:amalgamation}. Recall that every compression body $\compbody$ can be obtained by first taking the thickened version $\partial_-\compbody \times [0,1]$ of its lower boundary $\partial_-\compbody$ and then attaching some $1$-handles to $\partial_-\compbody \times \{1\}$ (see steps P1 and P2 in Figure \ref{fig:compbody}). In our example $\partial_-\compbodytwo_1  = \altaltsurface = \partial_-\compbodyone_2$, so $\compbodytwo_1$ can be built from $\altaltsurface \times [-1,0]$ by attaching two $1$-handles $h^{(1)}_1$ and $h^{(1)}_2$ along $\altaltsurface \times \{-1\}$. Similarly, $\compbodyone_2$ is constructed by taking $\altaltsurface \times [0,1]$ and attaching the $1$-handles $h^{(2)}_1$ and $h^{(2)}_2$ to $\altaltsurface \times \{1\}$.

The {\em amalgamation} process consists of two steps:
\begin{enumerate*}
\itemsep0em
	\item Collapse $\altaltsurface \times [-1,1]$ to $\altaltsurface \times \{0\}$, such that the attaching sites of the $1$-handles $h^{(1)}_1$, $h^{(1)}_2$,  $h^{(2)}_1$ and $h^{(2)}_2$ remain pairwise disjoint. (This can be achieved by slightly deforming the attaching maps, if necessary.)
	\item Set $\compbodyone = \compbodyone_1 \cup  h^{(2)}_1 \cup h^{(2)}_2$ and $\compbodytwo = \compbodytwo_2 \cup  h^{(1)}_1 \cup h^{(1)}_2$, see Figure \ref{fig:amalgamation}(ii).
\end{enumerate*}

\medskip

\begin{figure}[ht]
	\centering
	\includegraphics[scale=\MyFigScale]{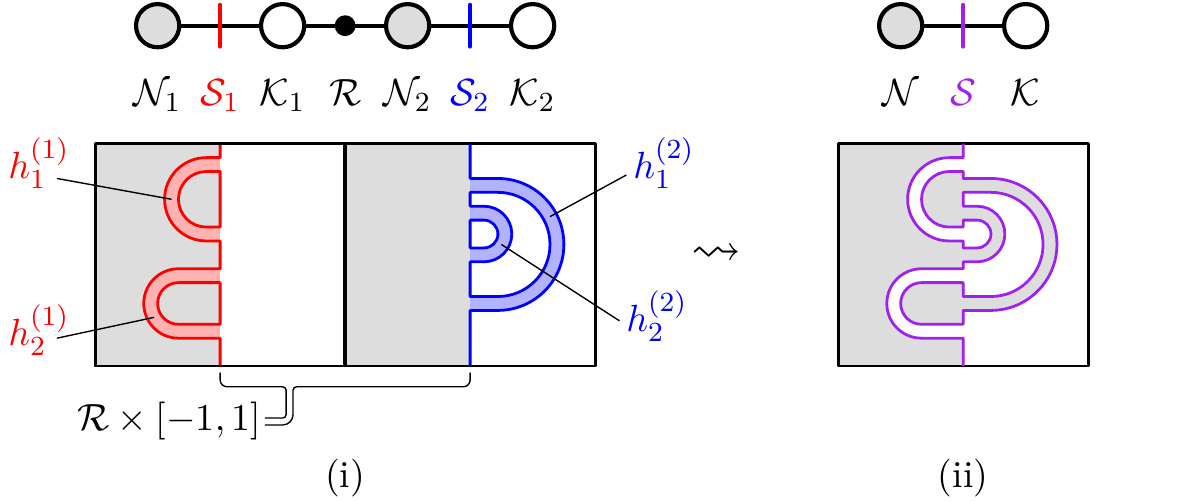}
	\caption{Amalgamating a generalized Heegaard splitting into a Heegaard splitting.}
	\label{fig:amalgamation}
\end{figure}

If $\altaltsurface$ is connected, then for the genus of the amalgamated Heegaard surface $\surface$ we have
\begin{align}
	g(\surface) = g(\surface_1) + g(\surface_2) - g(\altaltsurface).
	\label{eq:amalgamation_example}
\end{align}

However, in case $\altaltsurface$ has multiple connected components, then \eqref{eq:amalgamation_example} does not hold anymore. The procedure of amalgamation nevertheless works for arbitrary generalized Heegaard splittings (cf.\ Remark \ref{rem:generalized_heegaard}), and the formula \eqref{eq:amalgamation_example} can be adapted to the general setting as follows, by taking into account the Euler characteristic of the dual graph of the decomposition.

\begin{theorem}[Quantitative Amalgamation; cf.\ Theorems 2.8 and 2.9 in \cite{bachman2017computing}]~
\label{thm:amalgamation}
\begin{enumerate}
	\item Any generalized Heegaard splitting $\manifold = \bigcup_{i \in I}(\compbodyone_i \cup_{\surface_i}\compbodytwo_i)$ of a given $3$-manifold $\manifold$ can be amalgamated to a (classical) Heegaard splitting $\manifold = \compbodyone \cup_{\surface} \compbodytwo$ thereof.
	\item Let $\mathscr{D}$ be the decomposition $\manifold = \bigcup_{i \in I}\manifold_i$ underlying the generalized Heegaard splitting above, and $\Gamma(\mathscr{D})=(I,E)$ be its dual graph with Euler characteristic $\chi(\Gamma(\mathscr{D}))$. For any $e = \{u,v\} \in E$, let $\altaltsurface_e$  be the connected component of $\manifold_u \cap \manifold_v$ dual to $e$.\footnote{Note that there might be multiple arcs between the nodes $u$ and $v$ in $\Gamma(\mathscr{D})$. We account for all.} Then the genus $g(\surface)$ of the amalgamated Heegaard surface $\surface$ satisfies
\begin{align}
	g(\surface) = \sum_{i \in I}g(\surface_i) - \sum_{e \in E}g(\altaltsurface_e)+1-\chi(\Gamma(\mathscr{D})).
	\label{eq:amalgamation}
\end{align}
\end{enumerate}
\end{theorem}

\begin{remark}
\label{rem:generalized_heegaard}
In the definition of generalized Heegaard splittings, ordering the vertices of $\Gamma(\mathscr{D})$ and choosing the Heegaard splittings of the $\manifold_i$ in a compatible way might seem to be an ad-hoc requirement. However, this property arises naturally, when a generalized Heegaard splitting is constructed from a sequence of handle attachments. It also ensures that a generalized Heegaard splitting can always be amalgamated into a classical one. This feature lies at the heart of many applications, including the main result of \cite{bachman2017computing} according to which the problem of computing the Heegaard genus is \textbf{NP}-hard. We also make great use of the amalgamation procedure in Section \ref{sec:proof} to establish Theorem \ref{thm:pw-vol}.
\end{remark}

\subsection{Hyperbolic 3-manifolds}
\label{ssec:hyperbolic}

Hyperbolic geometry has been playing a role in the study of 3-manifolds for over a century \cite{mcmullen2011evolution}, but it rose into particular prominence after Thurston formulated the {\em geometrization conjecture} \cite{thurston1982three}, famously resolved by Perelman twenty years later \cite{perelman2002entropy, perelman2003ricci} (cf.\ \cite{porti2008geometrization}). Hyperbolic 3-manifolds constitute the richest family among geometric 3-manifolds, and, to this date, they remain the least understood. We refer to \cite{martelli2016introduction} for an introduction to this area.

A 3-manifold $\manifold$ is {\em hyperbolic}, if its interior can be obtained as a quotient of the hyperbolic 3-space $\mathbb{H}^3$ by a discrete group of isometries acting freely on $\mathbb{H}^3$. Equivalently, if the interior of $\manifold$ admits a complete Riemannian metric of constant sectional curvature $-1$. Throughout this section, $\manifold$ is assumed to be an orientable Riemannian 3-manifold. After fixing an orientation on $\manifold$, its metric tensor induces a ``volume form'' $\omega$. This in turn leads to the notion of {\em volume} defined via the integral
\[
	\vol{\mathcal{D}} = \int_\mathcal{D} \omega
\]
for any open set $\mathcal{D} \subseteq \manifold$. Also, any submanifold of $\manifold$ admits a Riemannian metric induced by the metric tensor of $\manifold$. Thus we may measure lengths of paths and areas of surfaces in $\manifold$ as well. We refer to \cite[Section 1.2]{martelli2016introduction} for details.

If $\manifold$ is compact, then $\vol{\manifold}$ is finite. The next striking result has been of paramount importance in geometric topology, as it says that ``geometric properties'' of finite-volume hyperbolic 3-manifolds are actually topological invariants.

\begin{theorem}[Mostow Rigidity Theorem \cite{mostow1968quasi}, cf.\ {\cite[Theorem 1.7.1]{aschenbrenner2015manifold}}, {\cite[Chapter 13]{martelli2016introduction}}]
\label{thm:rigidity}
Let $\manifold$ and $\altmanifold$ be finite-volume hyperbolic $3$-manifolds. Every isomorphism $\pi_1(\manifold) \rightarrow \pi_1(\altmanifold)$ between the fundamental groups of $\manifold$ and $\altmanifold$ is induced by a unique isometry $\manifold \rightarrow \altmanifold$.
\end{theorem}

\begin{corollary}
\label{cor:vol}
If two hyperbolic $3$-manifolds have different volume, then they cannot be homotopy equivalent, hence they cannot be homeomorphic.
\end{corollary}

\paragraph{Thick-thin decompositions.} As mentioned in the \nameref{sec:intro}, a key ingredient in the proof of Theorem \ref{thm:pw-vol} is the the {\em thick-thin decomposition theorem}, a fundamentally important structural result for hyperbolic manifolds of any dimension. In order to formulate it, we need to introduce the {\em injectivity radius} of a Riemannian manifold.

\begin{definition}[injectivity radius] Let $\manifold$ be a Riemannian manifold and $x \in \manifold$. The {\em injectivity radius of $\manifold$ at $x$}, denoted $\inj_x(\manifold)$, is the supremal value $r  > 0$ such that the metric ball of radius $r$ around $x$ is embedded in $\manifold$. The {\em injectivity radius of $\manifold$} is defined as the infimal value of $\inj_x(\manifold)$, i.e., $\inj(\manifold) = \inf\{\inj_x(\manifold) : x\in \manifold\}$.
\label{def:injrad}
\end{definition}

After fixing some threshold $\varepsilon> 0$, a Riemannian manifold $\manifold$ naturally decomposes into an {\em $\varepsilon$-thick} and an {\em $\varepsilon$-thin} part based on the injectivity radius of its points:
\begin{align}
	\manifold_{[\varepsilon,\infty)} = \{x \in \manifold : \inj_x(\manifold) \geq \varepsilon/2\} \quad \text{and} \quad \manifold_{(0,\varepsilon]} = \{x \in \manifold : \inj_x(\manifold) \leq \varepsilon/2\}.
\label{eq:thickthin}
\end{align}

We are now in the position to state the thick-thin decomposition theorem, according to which, for a sufficiently small constant $\varepsilon > 0$ only depending on the dimension $d$, the $\varepsilon$-thin part of any orientable hyperbolic $d$-manifold has a well-understood structure.\footnote{The manifolds in consideration are also required to be {\em complete} (as metric spaces). However, the way we define hyperbolic $d$-manifolds (i.e., quotients of $\mathbb{H}^d$ under discrete groups of isometries acting freely) automatically ensures their completeness.}

\begin{theorem}[Thick-Thin Decomposition; cf.\ {\cite[Chapter 4]{martelli2016introduction}}, {\cite[Section 5.3]{purcell2020book}}]
There exists a universal constant $\varepsilon_d > 0$, depending only on the dimension $d$, such that for any $\varepsilon \in (0,\varepsilon_d]$, the $\varepsilon$-thin part of any orientable hyperbolic $d$-manifold $\manifold$ consists of tubes around short geodesics diffeomorphic to $\nsphere{1} \times \mathbb{D}^{d-1}$, or cusps.\footnote{A $d$-dimensional cusp is a $d$-manifold with boundary that is diffeomorphic to $\mathcal{N} \times [0,\infty)$, where $\mathcal{N}$ is a $(d-1)$-dimensional flat, i.e., Euclidean, manifold. See \cite[Section 4.1]{martelli2016introduction} for a precise definition.}
\label{thm:thickthin}
\end{theorem}

\begin{figure}[ht]
	\centering
	\includegraphics[scale=\MyFigScale]{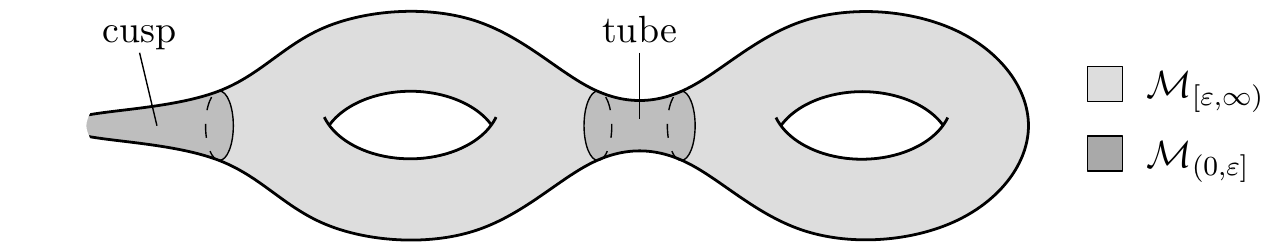}
	\caption{Thick-thin decomposition of a non-compact hyperbolic surface.}
	\label{fig:thickthin}
\end{figure}

\begin{remark} We conclude with some remarks about the thick-thin decomposition.
\label{rem:thickthin}
\begin{enumerate}
	\item In case of compact 3-manifolds, there are no cusps in the thick-thin decomposition, but only tubes. In dimension three, they are homeomorphic to solid tori. This is important, as Theorem \ref{thm:heeg-vol} is concerned with closed (hence compact) 3-manifolds.
	\label{rem:thickthin_1}
	\item The supremum of all $\varepsilon_d$ for which the conclusion of Theorem \ref{thm:thickthin} holds is called the $d$-dimensional {\em Margulis constant.} As of now, the precise value of $\varepsilon_d$ remains unknown. For $d=3$, it is known that $0.104 \leq \varepsilon_3 \leq 0.616$, cf.\ \cite[p.\ 92]{purcell2020book}.
	\label{rem:thickthin_2}
	\item Theorem \ref{thm:thickthin} is a corollary of a more general result about discrete subgroups of Lie groups, called the Margulis Lemma, appeared in \cite{kazhdan1968selberg}, cf.\ \cite[Section 4.2]{martelli2016introduction} and \cite[Theorem 5.22]{purcell2020book}.
	\label{rem:thickthin_3}
\end{enumerate}
\end{remark}

\section{Combinatorial width parameters for 3-manifolds}
\label{sec:comb}

Two important topological invariants we have already discussed are the Heegaard genus and the volume. Here we introduce a simple scheme that can be used to turn any non-negative graph parameter into a topological invariant for compact 3-manifolds: Given a graph parameter $p\colon\mathscr{G}\rightarrow\mathbb{N}$, defined on the set $\mathscr{G}$ of finite (multi)graphs,  simply put
\begin{align}
	p(\manifold) = \min\{p(\dual(\tri)):\tri\text{~is~a~triangulation~of~}\manifold\}.
	\label{eq:scheme}
\end{align}

We call any 3-manifold invariant $p$ obtained this way a {\em combinatorial width parameter}. For reasons explained in the \nameref{sec:intro}, in what follows, we apply this scheme on two notable graph parameters that have been playing a central role in the development of parameterized algorithms \cite{cygan2015parameterized, downey1999parameterized, downey2013fundamentals} and in structural graph theory \cite{bodlaender1994tourist, bodlaender2005discovering, kawarabayashi2007graph, lovasz2006graph}.

\paragraph{Treewidth and pathwidth of a graph.} Introduced by Robertson and Seymour \cite{robertson1983graph, robertson1986graph}, the treewidth and pathwidth informally measure how tree-like or path-like a graph is. To precisely define them, we first need to talk about a {\em tree decomposition} of a graph $G=(V,E)$: it is a pair $\mathscr{T} = (\{B_i:i \in I\},\alttree=(I,F))$ with {\em bag}s $B_i \subseteq V$ and a tree $\alttree=(I,F)$, such that
\begin{enumerate}
	\item \label{twpropone} $\bigcup_{i \in I} B_i = V$,
	\item \label{twproptwo} for every arc $\{u,v\} \in E$, there exists $i \in I$ with $\{u,v\} \subseteq B_i$, and
	\item \label{twpropthree} for every node $v \in V$, $T_v = \{i \in I:v \in B_i\}$ spans a connected subtree of $\alttree$.
\end{enumerate}
See Figure \ref{fig:treewidth} for an illustration. The \textit{width} of a tree decomposition equals $\max_{i \in I}|B_i|-1$ and the \emph{treewidth} $\tw{G}$ is the smallest width of any tree decomposition of $G$, cf.\ Figure \ref{fig:tw-ex}.

A \emph{path decomposition} of  a graph $G$ is merely a tree decomposition for which the tree $\alttree$ is required to be a path. Similarly, the  \emph{pathwidth} $\pw{G}$ of a graph $G$ is the minimum width of any path decomposition of~$G$. From the definitions, $\tw{G}\leq\pw{G}$.

\begin{figure}[ht]
	\centering
	\includegraphics[scale=\MyFigScale]{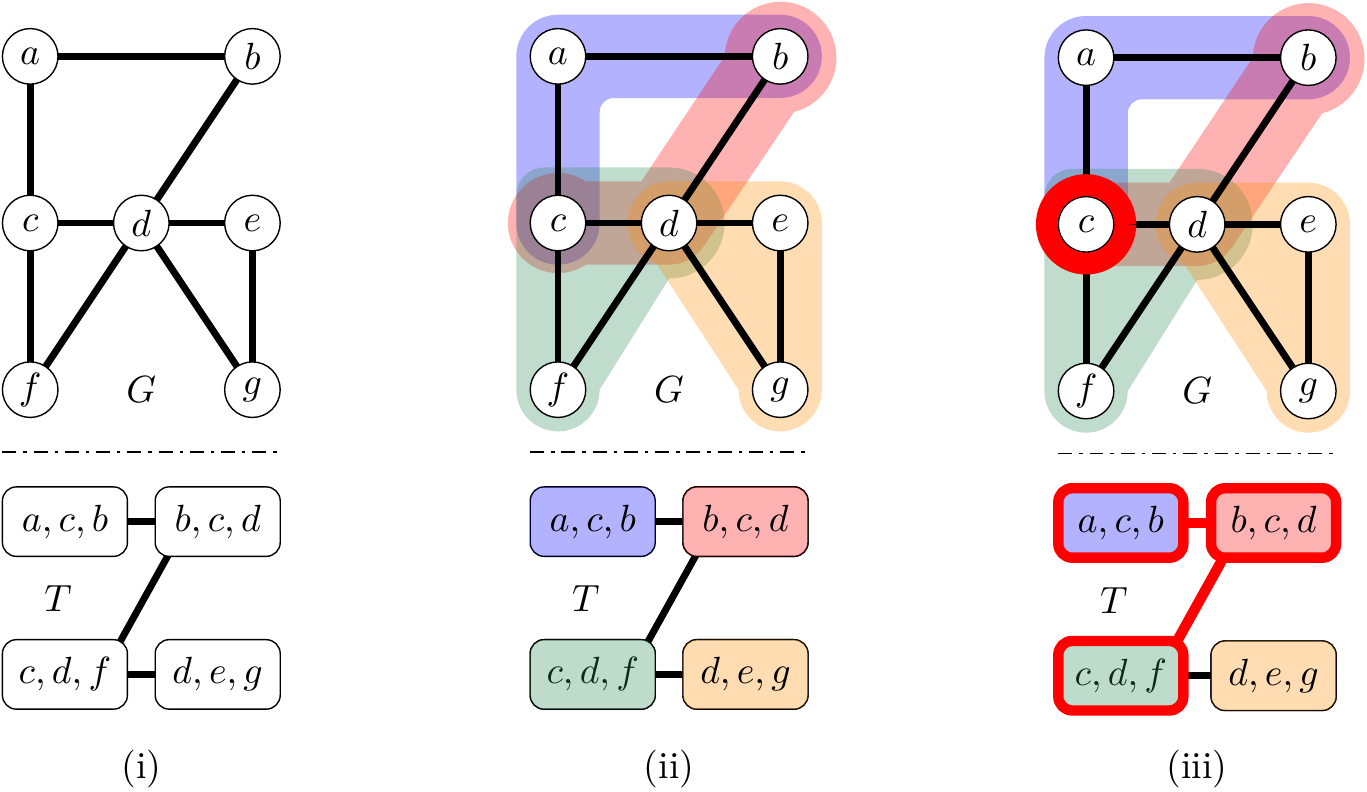}
	\caption{(i) A graph $G$ and a tree decomposition thereof, modeled on a tree $T$, of width 2. This tree decomposition also happens to be a path decomposition as $T$ is a path. (ii) Illustration of properties \ref{twpropone} and \ref{twproptwo} from the definition of a tree decomposition: The union of all bags equals $V$, and for every arc in $E$ there is a bag containing that arc. (iii) Illustration of property \ref{twpropthree}: The bags containing a given node (in this case $c$) span a connected subtree of $T$.}
	\label{fig:treewidth}
\end{figure}

\begin{figure}[!ht]
\centering
\includegraphics[scale=\MyFigScale]{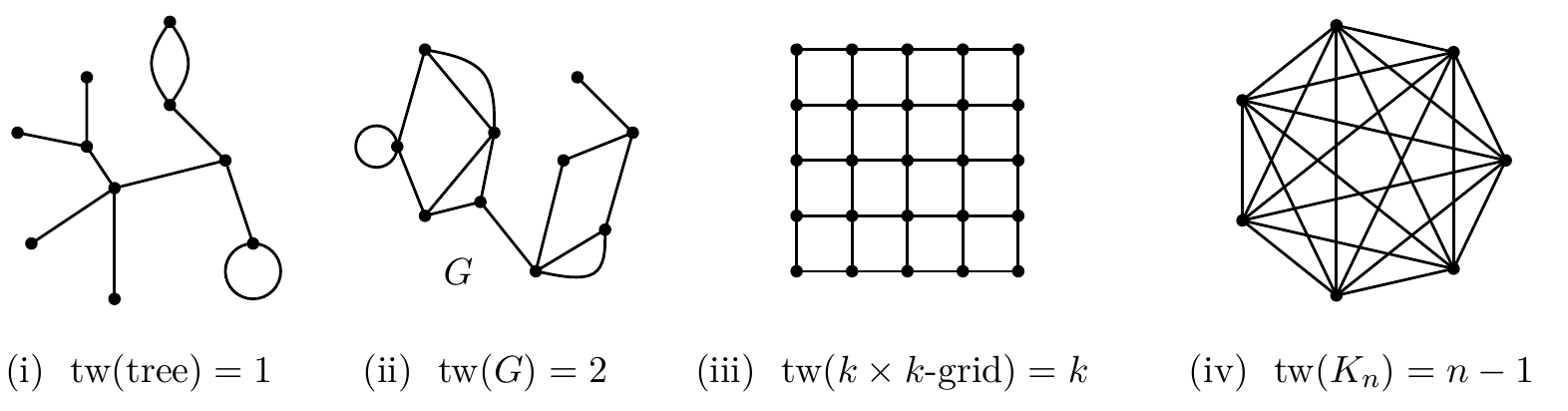}
\caption{(i) Graphs of treewidth one are precisely the trees (possibly with loops or multiarcs). (ii) A graph of treewidth two. (iii) The $k \times k$-grid (above $k =5$) has treewidth and pathwidth equal to $k$, thus planar graphs can have arbitrary large treewidth. (iv) The complete graph $K_n$ has treewidth $n-1$.}
\label{fig:tw-ex}
\end{figure}

\paragraph{Treewidth and pathwidth of a 3-manifold.} Having defined the treewidth and the pathwidth of a graph, based on \eqref{eq:scheme} it is immediate to extend these notions to 3-manifolds as
\begin{align}
	\begin{split}
		\tw{\manifold} &= \min\{\tw{\dual(\tri)}:\tri\text{~is~a~triangulation~of~}\manifold\}\text{,~and}\\
		\pw{\manifold} &= \min\{\pw{\dual(\tri)}:\tri\text{~is~a~triangulation~of~}\manifold\}.
	\end{split}
	\label{eq:tw-pw-mfd}
\end{align}
As $\tw{G}\leq\pw{G}$ for any graph $G$, we also have $\tw{\manifold}\leq\pw{\manifold}$ for any 3-manifold $\manifold$.

Recently, the quantitative relationship between these (and related) parameters and other topological invariants has become the subject of intense research. In \cite[Theorem 4]{huszar2019treewidth} it was shown that, for any closed, irreducible, non-Haken 3-manifold $\manifold$, the treewidth $\tw{\manifold}$ is bounded below in terms of the Heegaard genus $\heeg{\manifold}$ by means of the following inequality:
\begin{align}
	\heeg{\manifold} \leq 18(\tw{\manifold} + 1).
	\label{eq:heeg-tw}
\end{align}
The proof of \eqref{eq:heeg-tw} relies on the theory of generalized Heegaard splittings (see \cite[Section 6]{huszar2019treewidth} for details) and, in combination with work of Agol \cite{agol2003small}, implies the existence of 3-manifolds with arbitrary large treewidth.\footnote{A similar result was obtained in \cite{mesmay2019treewidth}. Here the treewidth of a knot diagram was related to the {\em sphere number} of the underlying knot, giving the first examples of knots where any diagram has high treewidth.} In subsequent work \cite{huszar2019manifold} it was proven (based on the theory of layered triangulations \cite{jaco2006layered}) that a reverse inequality holds for all closed 3-manifolds.
\begin{theorem}
\label{thm:pw-heeg}
For every closed $3$-manifold $\manifold$ with treewidth $\tw{\manifold}$, pathwidth $\pw{\manifold}$, and Heegaard genus $\heeg{\manifold}$ we have
\begin{align}
	\tw{\manifold} \leq \pw{\manifold} \leq 4\heeg{\manifold} - 2.
	\label{eq:heeg-pw}
\end{align}
\end{theorem}
Here the first inequality follows from the definitions of pathwidth and treewidth (see above); for details about the second one, we refer to \cite[Theorems 1.6 and 2.4]{huszar2020combinatorial} and \cite[Chapter 5]{huszar2020combinatorial}, where further refinements of these results are discussed along with related work by others.

\section{The proof of Theorem 1}
\label{sec:proof}

In this section we put together the ingredients discussed before, in order to prove that the volume of a closed hyperbolic 3-manifold provides a linear upper bound on its pathwidth.

The proof of Theorem \ref{thm:pw-vol} rests on Theorem \ref{thm:heeg-vol}, a folklore result according to which the Heegaard genus of a closed, orientable, hyperbolic 3-manifold can be upper-bounded in terms of its volume (see, e.g., \cite[p.\ 336--337]{shalen2007hyperbolic} or \cite{purcell2017independence}).

\begin{theorem}
\label{thm:heeg-vol}
There exists a universal constant $C''>0$ such that, for any closed, orientable, hyperbolic 3-manifold $\manifold$ with Heegaard genus $\heeg{\manifold}$ and volume $\vol{\manifold}$, we have
\begin{align}
	\heeg{\manifold} \leq C'' \cdot \vol{\manifold}.
	\label{eq:heeg-vol}
\vspace{-\topsep}
\end{align}
\end{theorem}

\begin{proof}[Proof of Theorem \ref{thm:pw-vol} assuming Theorem \ref{thm:heeg-vol}]
By combining \eqref{eq:heeg-vol} with the second inequality in \eqref{eq:heeg-pw} the statement of Theorem \ref{thm:pw-vol} is readily deduced.
\end{proof}

We are left with proving Theorem \ref{thm:heeg-vol}. As we were unable to locate a proof of this result in the literature, below we give a proof ourselves.

\begin{proof}[Proof of Theorem \ref{thm:heeg-vol}]
First, we describe the general strategy. Given a closed, orientable, hyperbolic 3-manifold $\manifold$, we start by taking a thick-thin decomposition of $\manifold$. By a theorem that goes back to J{\o}rgensen and Thurston, the thick part can always be triangulated using $O(\vol{\manifold})$ tetrahedra.
Next, we show that such a triangulation of the thick part lends itself to a generalized Heegaard splitting of $\manifold$, where the sum of genera of the Heegaard surfaces is $O(\vol{\manifold})$. In the final step, we amalgamate this generalized Heegaard splitting into a classical Heegaard splitting of $\manifold$, and show that its genus is $O(\vol{\manifold})$.

\bigskip

We now elaborate on the details. First, we invoke the aforementioned theorem by J{\o}rgensen--Thurston, carefully proved by Kobayashi--Rieck. To precisely state this result, let us define the {\em (closed) $\delta$-neighborhood} $\ov{N_\delta}\left(\mathcal{X}\right)$ of a subset $\mathcal{X}$ of a Riemannian 3-manifold $\manifold$ to be the set of those points in $\manifold$ that have distance at most $\delta$ from some point in $\mathcal{X}$.

\medskip

\begin{theorem}[J{\o}rgensen--Thurston {\cite[\S 5.11]{thurston2002geometry}}, Kobayashi--Rieck \cite{kobayashi2011linear}]
Let $\varepsilon \in (0,\varepsilon_3]$, where $\varepsilon_3$ is the Margulis constant in dimension three {\em (cf.\ Remark \ref{rem:thickthin}/\ref{rem:thickthin_2})}.
\begin{enumerate}
	\item For any $\delta > 0$, there exists a constant $K > 0$, depending on $\varepsilon$ and $\delta$, so that for any finite-volume hyperbolic $3$-manifold $\manifold$, the $\delta$-neighborhood $\ov{N_\delta}\left(\manifold_{[\varepsilon,\infty)}\right) \subset \manifold$ of the thick part $\manifold_{[\varepsilon,\infty)}$ admits a triangulation with at most $K\cdot\vol{\manifold}$ tetrahedra. \label{thm:jorgensen-thurston_1}
	\item Moreover, $\ov{N_\delta}\left(\manifold_{[\varepsilon,\infty)}\right)$ is obtained from $\manifold$ by removing open tubular neighborhoods around short geodesics, and truncating cusps {\em \cite[Proposition 1.2]{kobayashi2011linear}}.\label{thm:jorgensen-thurston_2}
\end{enumerate}
\label{thm:jorgensen-thurston}
\end{theorem}

Now we fix an $\varepsilon \in (0,\varepsilon_3]$ and some $\delta > 0$. Let $\manifold$ be a closed hyperbolic 3-manifold. 
In the work of Maria--Purcell \cite{maria2019treewidth}, Theorem \ref{thm:jorgensen-thurston} plays a crucial role in ensuring the treewidth of $\ov{N_\delta}\left(\manifold_{[\varepsilon,\infty)}\right)$ to be upper-bounded by a linear function of $\vol{\manifold}$, and that $\partial\ov{N_\delta}\left(\manifold_{[\varepsilon,\infty)}\right)$ can be filled with solid tori.
For proving Theorem \ref{thm:heeg-vol}, we utilize Theorem \ref{thm:jorgensen-thurston} differently:

\medskip

\begin{proposition}
Let $\mathcal{Y} = \ov{N_\delta}\left(\manifold_{[\varepsilon,\infty)}\right)$ as defined above. The following are true.
\begin{enumerate}[label=(\alph*)]
\itemsep0em
\item For the Heegaard genus of $\mathcal{Y}$ we have $\heeg{\mathcal{Y}} = O(\vol{\manifold})$. \label{cor:heeg_thick_1}
\item $\mathcal{Y}$ has $O(\vol{\manifold})$ boundary components, each of which are tori. \label{cor:heeg_thick_2}
\end{enumerate}
\label{cor:heeg_thick}
\end{proposition}

\renewcommand{\qedsymbol}{\ensuremath{\vartriangleleft}}

\begin{proof}[Proof of Proposition \ref{cor:heeg_thick}]
To establish \ref{cor:heeg_thick_1}, first consider a triangulation $\tri$ of $\mathcal{Y}$ with $O(\vol{\manifold})$ tetrahedra. Such a triangulation is guaranteed to exist by Theorem \ref{thm:jorgensen-thurston}/\ref{thm:jorgensen-thurston_1}. Fix an arbitrary partition $\mathscr{P} = \{\partial_1\mathcal{Y}, \partial_2\mathcal{Y}\}$ of the boundary components of $\mathcal{Y}$ (the trivial partition, i.e., $\partial_1\mathcal{Y} = \emptyset$, $\partial_2\mathcal{Y} = \partial\mathcal{Y}$, is also allowed). Follow a procedure similar to \cite[Theorem 2.1.11]{scharlemann2016lecture} to obtain a Heegaard splitting of $\mathcal{Y}$ compatible with $\mathscr{P}$. (For more details, see Example \ref{ex:heegaard_boundary} in Appendix \ref{app:heegaard}.) By construction, the genus of this splitting is $O(\vol{\manifold})$, hence, for the Heegaard genus of $\mathcal{Y}$, we have $\heeg{\mathcal{Y}} = O(\vol{\manifold})$. 
For the first part of \ref{cor:heeg_thick_2}, observe that, by passing to a first barycentric subdivision, we may assume a tetrahedron can contribute triangles to at most one boundary component. The second part of \ref{cor:heeg_thick_2} follows from Theorem \ref{thm:jorgensen-thurston}/\ref{thm:jorgensen-thurston_2}.
\end{proof}

\renewcommand{\qedsymbol}{\ensuremath{\square}}

As discussed in Section \ref{ssec:gen_heegaard}, any decomposition $\manifold = \bigcup_{i \in I}\manifold_i$ of a 3-manifold $\manifold$ into codimension zero submanifolds with pairwise disjoint interiors gives rise to generalized Heegaard splittings of $\manifold$. Because $\manifold$ is hyperbolic, this is also true for every thick-thin decomposition of $\manifold$. So let us proceed by taking a thick-thin decomposition
\begin{align}
	\mathscr{D} = \left\{ \manifold_i : i \in [m],~\textstyle\bigcup_{i =1}^m\manifold_i = \manifold,~\text{and}~\operatorname{int}(\manifold_i) \cap \operatorname{int}(\manifold_j) = \emptyset~\text{for}~i \neq j \right\}
	\label{eq:thickthin_proof}
\end{align}
of $\manifold$, where $\manifold_1 = \mathcal{Y} = \ov{N_\delta}\left(\manifold_{[\varepsilon,\infty)}\right)$ is the thick part and $\manifold_2\ldots,\manifold_m$ are the thin ones.

Note that, by Theorem \ref{thm:jorgensen-thurston}/\ref{thm:jorgensen-thurston_2}, each $\manifold_i$ ($2 \leq i \leq m$) is homeomorphic to a solid torus $\nsphere{1} \times \twodisk^2$, and $m = O(\vol{\manifold})$ by Proposition \ref{cor:heeg_thick}/\ref{cor:heeg_thick_2}. Let us label the nodes of $\Gamma(\mathscr{D})$ via the identity map. For each $i \in [m]$, we choose a Heegaard splitting $\manifold_i = \compbodyone_i \cup_{\surface_i}\compbodytwo_i$ of minimal genus compatible with this labeling. This gives a generalized Heegaard splitting $\manifold = \bigcup_{i \in I}(\compbodyone_i \cup_{\surface_i}\compbodytwo_i)$ of the hyperbolic 3-manifold $\manifold$. See Figure \ref{fig:thickthin_graph_splitting} for an illustration. 

\begin{figure}[!ht]
	\centering
	\includegraphics[scale=\MyFigScale]{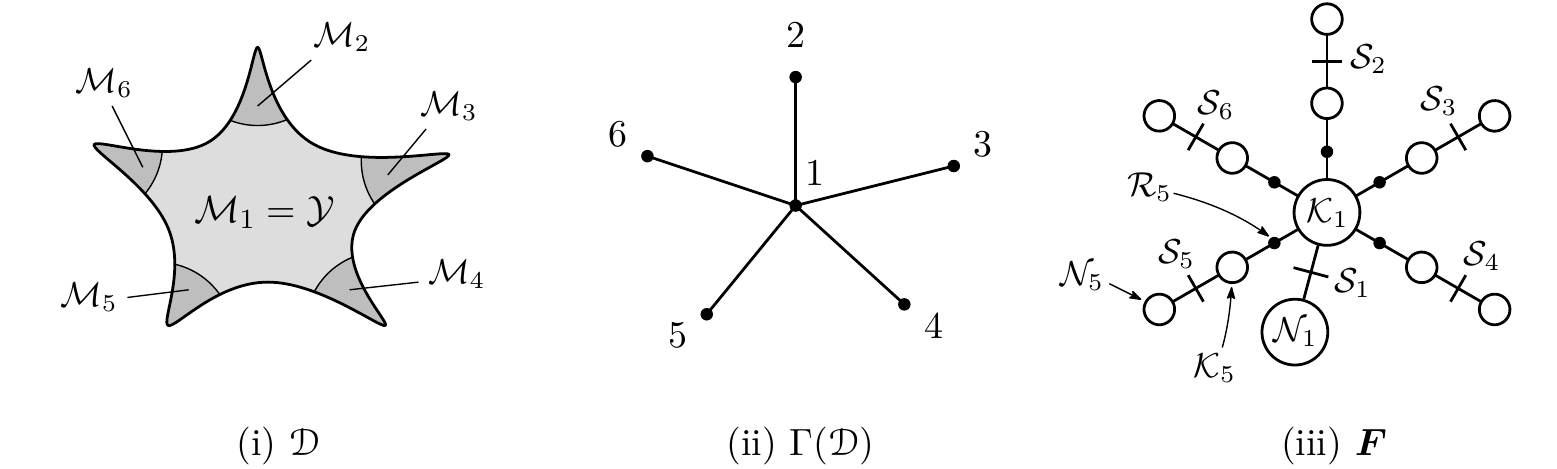}
	\caption{(i) Schematic example of a thick-thin decomposition $\mathscr{D}$ of a hyperbolic 3-manifold $\manifold$. (ii) The dual graph $\Gamma(\mathscr{D})$ of $\mathscr{D}$ with its nodes labeled via the identity map. (iii) The fork complex $\forkcomp$ of a generalized Heegaard splitting associated with $\mathscr{D}$ and the given labeling of $V(\Gamma(\mathscr{D}))$}
	\label{fig:thickthin_graph_splitting}	
\end{figure}
 
\begin{claim}
\label{claim:thickthin}
From the construction it directly follows that the generalized Heegaard splitting $\manifold = \bigcup_{i \in I}(\compbodyone_i \cup_{\surface_i}\compbodytwo_i)$ described above has the following properties:
\begin{enumerate}
	\item All the $\compbodyone_i$ are handlebodies. For $\compbodyone_1$ we have $g(\partial\compbodyone_1) = g(\surface_1) = O(\vol{\manifold})$.
	\item If $2 \leq i \leq m$, then $\compbodyone_i$ is a solid torus, therefore $g(\partial\compbodyone_i) = g(\surface_i) = 1$.
	\item $\compbodytwo_1$ is a compression body with $\partial_+\compbodytwo_1 = \surface_1$ and $\partial_-\compbodytwo_1 = \text{``disjoint union of $m$ tori.''}$
	\item If $2 \leq i \leq m$, then $\compbodytwo_i$ is a trivial compression body homeomorphic to $\torus^2 \times [0,1]$. For its boundary components we have $\partial_+\compbodytwo_i = \surface_i$ and $\partial_-\compbodytwo_i = \altaltsurface_i = \manifold_i \cap \manifold_1$.
	\item For the sum of the genera of the surfaces $\surface_i$ we have $\sum_{i=1}^m g(\surface_i) = O(\vol{\manifold})$. \hfill $\vartriangleleft$
\end{enumerate}
\end{claim}

\bigskip

By Theorem \ref{thm:amalgamation}, we may amalgamate this to a classical Heegaard splitting $\manifold = \compbodyone \cup_{\surface} \compbodytwo$.
Finally, by combining the data from Claim \ref{claim:thickthin} with the formula \eqref{eq:amalgamation} in Theorem \ref{thm:amalgamation}, we get

\begin{align}
	g(\surface) &= \sum_{i \in I}g(\surface_i) - \sum_{e \in E}g(\altaltsurface_e)+1-\chi(\Gamma(\mathscr{D}))\\
	&= g(\surface_1) + \sum_{i=2}^m\left(g(\surface_i) - g(\altaltsurface_i)\right) + 1 -\chi(\Gamma(\mathscr{D}))\\
	&= O(\vol{\manifold}) + \sum_{i=2}^m\left(1 - 1\right) + 1 - 1 =  O(\vol{\manifold}).
\end{align}

\bigskip

This concludes the proof of Theorem \ref{thm:heeg-vol}.
\end{proof}

\section{Discussion}
\label{sec:discussion}

\paragraph{Pathwidth vs.\ treewidth vs.\ volume.} The inequalities \eqref{eq:tw-vol} and \eqref{eq:pw-vol} respectively provide information about the quantitative relationship between the treewidth and the volume, and the pathwidth and the volume of hyperbolic 3-manifolds. It is natural to study the sharpness of these inequalities both in absolute terms and also relative to each other.

In \cite[Section 6]{maria2019treewidth} Maria and Purcell show that, by performing appropriate {\em Dehn fillings} on hyperbolic {\em $2$-bridge knot exteriors}, one can obtain an infinite family of closed hyperbolic 3-manifolds with bounded treewidth, but unbounded volume. The bound on the treewidth is established through a construction of small-treewidth triangulations of these manifolds, based on the work of Sakuma--Weeks \cite{sakuma1995examples} on triangulating 2-bridge knot exteriors. 
It is not difficult to see that these triangulations have bounded pathwidth, too.

Regarding the comparison of \eqref{eq:tw-vol} and \eqref{eq:pw-vol}, recall that, from the definitions of pathwidth and treewidth (Section \ref{sec:comb}) if follows that $\tw{\manifold} \leq \pw{\manifold}$ for every 3-manifold $\manifold$. However, while there are examples for which both of these quantities are small \cite{huszar2019manifold} or arbitrary large \cite{huszar2019treewidth}, we do not know whether their {\em difference} can be arbitrary large. Thus we ask:

\begin{question}
\label{que:pw-tw}
Can the difference $\pw{\manifold} - \tw{\manifold}$ be arbitrary large?
\end{question}

\noindent
Note that, in case of graphs, trees can have arbitrary large pathwidth, e.g., the {\em complete binary tree} $T_h$ of height $h$ satisfies $\pw{T_h}=\lceil h/2 \rceil$, cf.\ \cite[Theorem 67]{bodlaender1998partial} or \cite[Lemma 2.8]{takahashi1994minimal}. 

\paragraph{Algorithmic aspects of small pathwidth.} In the \nameref{sec:intro} it was mentioned that input triangulations with small-pathwidth dual graphs can speed up algorithms that are fixed-parameter tractable (FPT) in the treewidth. Here we briefly elaborate on this.

A typical FPT-algorithm $\mathscr{A}$ exploits the small treewidth of the dual graph of its input triangulation by using a data structure called a {\em nice tree decomposition}. This is a particular kind of tree decomposition, whose bags can be grouped into three different types: {\em forget}, {\em introduce}, and {\em join} bags. A triangulation $\tri$ with $n$ tetrahedra and $\tw{\dual(\tri)} = k$ always admits a nice tree decomposition with at most $4n$ bags of width $k$, see \cite[Section 13.1]{kloks1994treewidth}.

Upon taking such a triangulation $\tri$ as input, the algorithm $\mathscr{A}$ would first construct a tree decomposition $\mathscr{T}$ of $\dual(\tri)$ of small width, turn $\mathscr{T}$ into a {\em nice} tree decomposition $\mathscr{T}_{\mathrm{nice}}$ (which is still of small width, as noted above), then parse $\mathscr{T}_{\mathrm{nice}}$ and perform its specific computation at each bag thereof. Depending on the problem to be solved, processing a {\em join} bag can be orders of magnitudes slower than processing an {\em introduce} or {\em forget} bag (we refer to \cite[Appendix C]{huszar2019treewidth} for more details, and to \cite[Section 4(d)]{burton2016parameterized} for a real-life example).

Now, if the decomposition $\mathscr{T}$ happens to be a path decomposition (i.e., a tree decomposition where the underlying tree is a path), then the procedure for constructing $\mathscr{T}_{\mathrm{nice}}$ results in a nice tree decomposition {\em without} join bags. Therefore, if the pathwidth $\pw{\dual(\tri)}$ is not ``much'' larger than the treewidth $\tw{\dual(\tri)}$, constructing a path decomposition (instead of an arbitrary tree decomposition) in the first step can potentially be very beneficial for the overall running time of $\mathscr{A}$, as the algorithm does not have to deal with join bags at all. 

\paragraph{Topological parameters for FPT-algorithms.} Algorithms that are FPT in the treewidth have been very successful in 3-dimensional topology. They all come with a caveat though: their fast execution presumes an input triangulation whose dual graph has small treewidth. However, in case the triangulation at hand has high treewidth, finding another triangulation of the same 3-manifold that has smaller treewidth might be very difficult.

To address this challenge, recently there has been a growing interest in researching algorithms that are FPT in {\em topological} parameters (e.g., the first Betti number \cite{maria2019polynomial}), that do not depend on the particular input triangulation, but only on the underlying 3-manifold.

Together with \cite{maria2019treewidth}, our work reinforces the potential of the volume of becoming a useful topological parameter for FPT-algorithms in the realm of hyperbolic 3-manifolds.

\bibliography{references}

\begin{thebibliography}{10}

\bibitem{agol2003small}
I.~Agol.
\newblock Small 3-manifolds of large genus.
\newblock {\em Geom. Dedicata}, 102:53--64, 2003.
\newblock \href {https://dx.doi.org/10.1023/B:GEOM.0000006584.85248.c5}
  {\path{doi:10.1023/B:GEOM.0000006584.85248.c5}}, \href
  {https://mathscinet.ams.org/mathscinet-getitem?mr=2026837}
  {\path{MR:2026837}}, \href {https://zbmath.org/?q=an:1039.57008}
  {\path{Zbl:1039.57008}}.

\bibitem{aschenbrenner2015manifold}
M.~Aschenbrenner, S.~Friedl, and H.~Wilton.
\newblock {\em 3-Manifold Groups}, volume~20 of {\em EMS Ser. Lect. Math.}
\newblock Eur. Math. Soc. (EMS), Z\"{u}rich, 2015.
\newblock \href {https://dx.doi.org/10.4171/154} {\path{doi:10.4171/154}},
  \href {https://mathscinet.ams.org/mathscinet-getitem?mr=3444187}
  {\path{MR:3444187}}, \href {https://zbmath.org/?q=an:1326.57001}
  {\path{Zbl:1326.57001}}.

\bibitem{bachman2017computing}
D.~Bachman, R.~Derby-Talbot, and E.~Sedgwick.
\newblock Computing {H}eegaard genus is {NP}-hard.
\newblock In {\em {A Journey Through Discrete Mathematics: A Tribute to
  Ji\v{r}\'i Matou\v{s}ek}}, pages 59--87. Springer, Cham, 2017.
\newblock \href {https://dx.doi.org/10.1007/978-3-319-44479-6}
  {\path{doi:10.1007/978-3-319-44479-6}}, \href
  {https://mathscinet.ams.org/mathscinet-getitem?mr=3726594}
  {\path{MR:3726594}}, \href {https://zbmath.org/?q=an:1388.57020}
  {\path{Zbl:1388.57020}}.

\bibitem{bing1959alternative}
R.~H. Bing.
\newblock An alternative proof that {$3$}-manifolds can be triangulated.
\newblock {\em Ann. Math. (2)}, 69:37--65, 1959.
\newblock \href {https://dx.doi.org/10.2307/1970092}
  {\path{doi:10.2307/1970092}}, \href
  {https://mathscinet.ams.org/mathscinet-getitem?mr=0100841}
  {\path{MR:0100841}}, \href {https://zbmath.org/?q=an:0106.16604}
  {\path{Zbl:0106.16604}}.

\bibitem{bodlaender1994tourist}
H.~L. Bodlaender.
\newblock A tourist guide through treewidth.
\newblock {\em Acta Cybern.}, 11(1-2):1--21, 1993.
\newblock
  \href{https://www.persistent-identifier.nl/urn:nbn:nl:ui:10-1874-2301}{\texttt{urn:nbn:nl:ui:10-1874-2301}}.
\newblock \href {https://mathscinet.ams.org/mathscinet-getitem?mr=1268488}
  {\path{MR:1268488}}, \href {https://zbmath.org/?q=an:0804.68101}
  {\path{Zbl:0804.68101}}.

\bibitem{bodlaender1998partial}
H.~L. Bodlaender.
\newblock A partial \emph{k}-arboretum of graphs with bounded treewidth.
\newblock {\em Theor. Comput. Sci.}, 209(1--2):1--45, 1998.
\newblock \href {https://dx.doi.org/10.1016/S0304-3975(97)00228-4}
  {\path{doi:10.1016/S0304-3975(97)00228-4}}, \href
  {https://mathscinet.ams.org/mathscinet-getitem?mr=1647486}
  {\path{MR:1647486}}, \href {https://zbmath.org/?q=an:0912.68148}
  {\path{Zbl:0912.68148}}.

\bibitem{bodlaender2005discovering}
H.~L. Bodlaender.
\newblock Discovering treewidth.
\newblock In {\em Proc. 31st Conf. Curr. Trends Theory Pract. Comput. Sci.
  ({SOFSEM} 2005)}, pages 1--16, 2005.
\newblock \href {https://dx.doi.org/10.1007/978-3-540-30577-4_1}
  {\path{doi:10.1007/978-3-540-30577-4_1}}, \href
  {https://zbmath.org/?q=an:1117.68451} {\path{Zbl:1117.68451}}.

\bibitem{burton2013regina}
B.~A. Burton.
\newblock Computational topology with {R}egina: algorithms, heuristics and
  implementations.
\newblock In {\em Geometry and Topology Down Under}, volume 597 of {\em
  Contemp. Math.}, pages 195--224. Amer. Math. Soc., Providence, RI, 2013.
\newblock \href {https://dx.doi.org/10.1090/conm/597/11877}
  {\path{doi:10.1090/conm/597/11877}}, \href
  {https://mathscinet.ams.org/mathscinet-getitem?mr=3186674}
  {\path{MR:3186674}}, \href {https://zbmath.org/?q=an:1279.57004}
  {\path{Zbl:1279.57004}}.

\bibitem{burton2014crushing}
B.~A. Burton.
\newblock A new approach to crushing 3-manifold triangulations.
\newblock {\em Discrete Comput. Geom.}, 52(1):116--139, 2014.
\newblock \href {https://dx.doi.org/10.1007/s00454-014-9572-y}
  {\path{doi:10.1007/s00454-014-9572-y}}, \href
  {https://mathscinet.ams.org/mathscinet-getitem?mr=3231034}
  {\path{MR:3231034}}, \href {https://zbmath.org/?q=an:1317.57012}
  {\path{Zbl:1317.57012}}.

\bibitem{burton2018homfly}
B.~A. Burton.
\newblock The {HOMFLY}-{PT} polynomial is fixed-parameter tractable.
\newblock In {\em 34th {I}nt. {S}ymp. {C}omput. {G}eom. ({SoCG} 2018)},
  volume~99 of {\em LIPIcs. Leibniz Int. Proc. Inform.}, pages 18:1--18:14.
  Schloss Dagstuhl--Leibniz-Zent. Inf., 2018.
\newblock \href {https://dx.doi.org/10.4230/LIPIcs.SoCG.2018.18}
  {\path{doi:10.4230/LIPIcs.SoCG.2018.18}}, \href
  {https://mathscinet.ams.org/mathscinet-getitem?mr=3824262}
  {\path{MR:3824262}}.

\bibitem{Regina}
B.~A. Burton, R.~Budney, W.~Pettersson, et~al.
\newblock Regina: Software for low-dimensional topology, 1999--2019.
\newblock Version 5.1.
\newblock URL: \url{https://regina-normal.github.io}.

\bibitem{burton2017courcelle}
B.~A. Burton and R.~G. Downey.
\newblock Courcelle's theorem for triangulations.
\newblock {\em J. Comb. Theory, Ser. {A}}, 146:264--294, 2017.
\newblock \href {https://dx.doi.org/10.1016/j.jcta.2016.10.001}
  {\path{doi:10.1016/j.jcta.2016.10.001}}, \href
  {https://mathscinet.ams.org/mathscinet-getitem?mr=3574232}
  {\path{MR:3574232}}, \href {https://zbmath.org/?q=an:1353.05122}
  {\path{Zbl:1353.05122}}.

\bibitem{burton2016parameterized}
B.~A. Burton, T.~Lewiner, J.~Paix{\~{a}}o, and J.~Spreer.
\newblock Parameterized complexity of discrete {M}orse theory.
\newblock {\em {ACM} Trans. Math. Softw.}, 42(1):6:1--6:24, 2016.
\newblock \href {https://dx.doi.org/10.1145/2738034}
  {\path{doi:10.1145/2738034}}, \href
  {https://mathscinet.ams.org/mathscinet-getitem?mr=3472422}
  {\path{MR:3472422}}, \href {https://zbmath.org/?q=an:1347.68165}
  {\path{Zbl:1347.68165}}.

\bibitem{burton2018algorithms}
B.~A. Burton, C.~Maria, and J.~Spreer.
\newblock Algorithms and complexity for {Turaev--Viro} invariants.
\newblock {\em J. Appl. Comput. Topol.}, 2(1--2):33--53, 2018.
\newblock \href {https://dx.doi.org/10.1007/s41468-018-0016-2}
  {\path{doi:10.1007/s41468-018-0016-2}}, \href
  {https://mathscinet.ams.org/mathscinet-getitem?mr=3873178}
  {\path{MR:3873178}}, \href {https://zbmath.org/?q=an:07089248}
  {\path{Zbl:07089248}}.

\bibitem{pettersson2014fixed}
B.~A. Burton and W.~Pettersson.
\newblock Fixed parameter tractable algorithms in combinatorial topology.
\newblock In {\em Proc. 20th Int. Conf. Comput. Comb. ({COCOON} 2014)}, pages
  300--311, 2014.
\newblock \href {https://dx.doi.org/10.1007/978-3-319-08783-2_26}
  {\path{doi:10.1007/978-3-319-08783-2_26}}, \href
  {https://mathscinet.ams.org/mathscinet-getitem?mr=3247596}
  {\path{MR:3247596}}, \href {https://zbmath.org/?q=an:1423.68205}
  {\path{Zbl:1423.68205}}.

\bibitem{burton2013complexity}
B.~A. Burton and J.~Spreer.
\newblock The complexity of detecting taut angle structures on triangulations.
\newblock In {\em Proc. 24th Annu. {ACM-SIAM} Symp. Discrete Algorithms ({SODA}
  2013)}, pages 168--183, 2013.
\newblock \href {https://dx.doi.org/10.1137/1.9781611973105.13}
  {\path{doi:10.1137/1.9781611973105.13}}, \href
  {https://mathscinet.ams.org/mathscinet-getitem?mr=3185388}
  {\path{MR:3185388}}, \href {https://zbmath.org/?q=an:1421.68161}
  {\path{Zbl:1421.68161}}.

\bibitem{cygan2015parameterized}
M.~Cygan, F.~V. Fomin, {\L}.~Kowalik, D.~Lokshtanov, D.~Marx, M.~Pilipczuk,
  M.~Pilipczuk, and S.~Saurabh.
\newblock {\em Parameterized Algorithms}.
\newblock Springer, Cham, 2015.
\newblock \href {https://dx.doi.org/10.1007/978-3-319-21275-3}
  {\path{doi:10.1007/978-3-319-21275-3}}, \href
  {https://mathscinet.ams.org/mathscinet-getitem?mr=3380745}
  {\path{MR:3380745}}, \href {https://zbmath.org/?q=an:1334.90001}
  {\path{Zbl:1334.90001}}.

\bibitem{mesmay2019treewidth}
A.~{de Mesmay}, J.~Purcell, S.~Schleimer, and E.~Sedgwick.
\newblock On the tree-width of knot diagrams.
\newblock {\em J. Comput. Geom.}, 10(1):164--180, 2019.
\newblock \href {https://dx.doi.org/10.20382/jocg.v10i1a6}
  {\path{doi:10.20382/jocg.v10i1a6}}, \href
  {https://mathscinet.ams.org/mathscinet-getitem?mr=3957223}
  {\path{MR:3957223}}, \href {https://zbmath.org/?q=an:1432.57017}
  {\path{Zbl:1432.57017}}.

\bibitem{derby-talbot2009stabilization}
R.~Derby-Talbot.
\newblock Stabilization, amalgamation and curves of intersection of {H}eegaard
  splittings.
\newblock {\em Algebr. Geom. Topol.}, 9(2):811--832, 2009.
\newblock \href {https://dx.doi.org/10.2140/agt.2009.9.811}
  {\path{doi:10.2140/agt.2009.9.811}}, \href
  {https://mathscinet.ams.org/mathscinet-getitem?mr=2505126}
  {\path{MR:2505126}}, \href {https://zbmath.org/?q=an:1176.57021}
  {\path{Zbl:1176.57021}}.

\bibitem{downey1999parameterized}
R.~G. Downey and M.~R. Fellows.
\newblock {\em Parameterized Complexity}.
\newblock Monogr. Comput. Sci. Springer-Verlag New York, 1999.
\newblock \href {https://dx.doi.org/10.1007/978-1-4612-0515-9}
  {\path{doi:10.1007/978-1-4612-0515-9}}, \href
  {https://mathscinet.ams.org/mathscinet-getitem?mr=1656112}
  {\path{MR:1656112}}, \href {https://zbmath.org/?q=an:0914.68076}
  {\path{Zbl:0914.68076}}.

\bibitem{downey2013fundamentals}
R.~G. Downey and M.~R. Fellows.
\newblock {\em Fundamentals of Parameterized Complexity}.
\newblock Texts Comput. Sci. Springer, London, 2013.
\newblock \href {https://dx.doi.org/10.1007/978-1-4471-5559-1}
  {\path{doi:10.1007/978-1-4471-5559-1}}, \href
  {https://mathscinet.ams.org/mathscinet-getitem?mr=3154461}
  {\path{MR:3154461}}, \href {https://zbmath.org/?q=an:1358.68006}
  {\path{Zbl:1358.68006}}.

\bibitem{heegaard1916analysis}
P.~Heegaard.
\newblock Sur l'"{A}nalysis situs".
\newblock {\em Bull. Soc. Math. France}, 44:161--242, 1916.
\newblock \href {https://dx.doi.org/10.24033/bsmf.968}
  {\path{doi:10.24033/bsmf.968}}, \href
  {https://mathscinet.ams.org/mathscinet-getitem?mr=1504754}
  {\path{MR:1504754}}.

\bibitem{hempel2004manifolds}
J.~Hempel.
\newblock {\em 3-Manifolds}.
\newblock AMS Chelsea Publ., Providence, RI, 2004.
\newblock Reprint of the 1976 original.
\newblock \href {https://dx.doi.org/10.1090/chel/349}
  {\path{doi:10.1090/chel/349}}, \href
  {https://mathscinet.ams.org/mathscinet-getitem?mr=2098385}
  {\path{MR:2098385}}, \href {https://zbmath.org/?q=an:1058.57001}
  {\path{Zbl:1058.57001}}.

\bibitem{huszar2020combinatorial}
K.~Husz\'ar.
\newblock {\em Combinatorial width parameters for 3-dimensonal manifolds}.
\newblock PhD thesis, IST Austria, June 2020.
\newblock \href {https://dx.doi.org/10.15479/AT:ISTA:8032}
  {\path{doi:10.15479/AT:ISTA:8032}}.

\bibitem{huszar2019manifold}
K.~Husz\'ar and J.~Spreer.
\newblock 3-{M}anifold triangulations with small treewidth.
\newblock In {\em 35th {I}nt. {S}ymp. {C}omput. {G}eom. ({SoCG} 2019)}, volume
  129 of {\em LIPIcs. Leibniz Int. Proc. Inf.}, pages 44:1--44:20. Schloss
  Dagstuhl--Leibniz-Zent. Inf., 2019.
\newblock \href {https://dx.doi.org/10.4230/LIPIcs.SoCG.2019.44}
  {\path{doi:10.4230/LIPIcs.SoCG.2019.44}}, \href
  {https://mathscinet.ams.org/mathscinet-getitem?mr=3968630}
  {\path{MR:3968630}}.

\bibitem{huszar2019treewidth}
K.~Husz\'ar, J.~Spreer, and U.~Wagner.
\newblock On the treewidth of triangulated 3-manifolds.
\newblock {\em J. Comput. Geom.}, 10(2):70--98, 2019.
\newblock \href {https://dx.doi.org/10.20382/jogc.v10i2a5}
  {\path{doi:10.20382/jogc.v10i2a5}}, \href
  {https://mathscinet.ams.org/mathscinet-getitem?mr=4039886}
  {\path{MR:4039886}}, \href {https://zbmath.org/?q=an:07150581}
  {\path{Zbl:07150581}}.

\bibitem{jaco1980lectures}
W.~Jaco.
\newblock {\em Lectures on Three-Manifold Topology}, volume~43 of {\em CBMS
  Reg. Conf. Ser. Math.}
\newblock Amer. Math. Soc., Providence, R.I., 1980.
\newblock \href {https://dx.doi.org/10.1090/cbms/043}
  {\path{doi:10.1090/cbms/043}}, \href
  {https://mathscinet.ams.org/mathscinet-getitem?mr=565450} {\path{MR:565450}},
  \href {https://zbmath.org/?q=an:0433.57001} {\path{Zbl:0433.57001}}.

\bibitem{jaco2003efficient}
W.~Jaco and J.~H. Rubinstein.
\newblock {$0$}-efficient triangulations of 3-manifolds.
\newblock {\em J. Differential Geom.}, 65(1):61--168, 2003.
\newblock \href {https://dx.doi.org/10.4310/jdg/1090503053}
  {\path{doi:10.4310/jdg/1090503053}}, \href
  {https://mathscinet.ams.org/mathscinet-getitem?mr=2057531}
  {\path{MR:2057531}}, \href {https://zbmath.org/?q=an:1068.57023}
  {\path{Zbl:1068.57023}}.

\bibitem{jaco2006layered}
W.~Jaco and J.~H. Rubinstein.
\newblock Layered-triangulations of 3-manifolds, 2006.
\newblock 97 pages, 32 figures.
\newblock \href {https://arxiv.org/abs/math/0603601}
  {\path{arXiv:math/0603601}}.

\bibitem{kazhdan1968selberg}
D.~A. Ka\v{z}dan and G.~A. Margulis.
\newblock A proof of {S}elberg's hypothesis.
\newblock {\em Math. USSR--Sbornik}, 4(1):147--152, 1968.
\newblock Translation from {\em Mat. Sb. (N.S.)}, 75(117):163--168, 1968.
  Translated by: Z.\ Skalsky.
\newblock \href {https://dx.doi.org/10.1070/SM1968v004n01ABEH002782}
  {\path{doi:10.1070/SM1968v004n01ABEH002782}}, \href
  {https://mathscinet.ams.org/mathscinet-getitem?mr=0223487}
  {\path{MR:0223487}}, \href {https://zbmath.org/?q=an:0241.22024}
  {\path{Zbl:0241.22024}}.

\bibitem{kawarabayashi2007graph}
K.~Kawarabayashi and B.~Mohar.
\newblock Some recent progress and applications in graph minor theory.
\newblock {\em Graphs Combin.}, 23(1):1--46, 2007.
\newblock \href {https://dx.doi.org/10.1007/s00373-006-0684-x}
  {\path{doi:10.1007/s00373-006-0684-x}}, \href
  {https://mathscinet.ams.org/mathscinet-getitem?mr=2292102}
  {\path{MR:2292102}}, \href {https://zbmath.org/?q=an:1114.05096}
  {\path{Zbl:1114.05096}}.

\bibitem{kloks1994treewidth}
T.~Kloks.
\newblock {\em Treewidth: Computations and Approximations}, volume 842 of {\em
  Lect. Notes Comput. Sci.}
\newblock Springer, 1994.
\newblock \href {https://dx.doi.org/10.1007/BFb0045375}
  {\path{doi:10.1007/BFb0045375}}, \href
  {https://mathscinet.ams.org/mathscinet-getitem?mr=1312164}
  {\path{MR:1312164}}, \href {https://zbmath.org/?q=an:0825.68144}
  {\path{Zbl:0825.68144}}.

\bibitem{kobayashi2011linear}
T.~Kobayashi and Y.~Rieck.
\newblock A linear bound on the tetrahedral number of manifolds of bounded
  volume (after {J}\o rgensen and {T}hurston).
\newblock In {\em Topology and Geometry in Dimension Three}, volume 560 of {\em
  Contemp. Math.}, pages 27--42. Amer. Math. Soc., Providence, RI, 2011.
\newblock \href {https://dx.doi.org/10.1090/conm/560/11089}
  {\path{doi:10.1090/conm/560/11089}}, \href
  {https://mathscinet.ams.org/mathscinet-getitem?mr=2866921}
  {\path{MR:2866921}}, \href {https://zbmath.org/?q=an:1335.57028}
  {\path{Zbl:1335.57028}}.

\bibitem{lovasz2006graph}
L.~Lov\'{a}sz.
\newblock Graph minor theory.
\newblock {\em Bull. Amer. Math. Soc. (N.S.)}, 43(1):75--86, 2006.
\newblock \href {https://dx.doi.org/10.1090/S0273-0979-05-01088-8}
  {\path{doi:10.1090/S0273-0979-05-01088-8}}, \href
  {https://mathscinet.ams.org/mathscinet-getitem?mr=2188176}
  {\path{MR:2188176}}, \href {https://zbmath.org/?q=an:1082.05082}
  {\path{Zbl:1082.05082}}.

\bibitem{maria2019parametrized}
C.~Maria.
\newblock Parameterized complexity of quantum invariants, 2019.
\newblock \href {https://arxiv.org/abs/1910.00477} {\path{arXiv:1910.00477}}.

\bibitem{maria2019treewidth}
C.~Maria and J.~Purcell.
\newblock Treewidth, crushing and hyperbolic volume.
\newblock {\em Algebr. Geom. Topol.}, 19(5):2625--2652, 2019.
\newblock \href {https://dx.doi.org/10.2140/agt.2019.19.2625}
  {\path{doi:10.2140/agt.2019.19.2625}}, \href
  {https://mathscinet.ams.org/mathscinet-getitem?mr=4023324}
  {\path{MR:4023324}}, \href {https://zbmath.org/?q=an:07142614}
  {\path{Zbl:07142614}}.

\bibitem{maria2019polynomial}
C.~Maria and J.~Spreer.
\newblock A polynomial-time algorithm to compute {T}uraev--{V}iro invariants
  $\operatorname{TV}_{4,q}$ of 3-manifolds with bounded first {Betti} number.
\newblock {\em Found. Comput. Math.}, pages 1--22, 2019.
\newblock Online: 11.11.2019.
\newblock \href {https://dx.doi.org/10.1007/s10208-019-09438-8}
  {\path{doi:10.1007/s10208-019-09438-8}}.

\bibitem{martelli2016introduction}
B.~Martelli.
\newblock {\em An Introduction to Geometric Topology}.
\newblock CreateSpace, 2016.
\newblock URL:
  \url{http://people.dm.unipi.it/martelli/geometric_topology.html}, \href
  {https://arxiv.org/abs/1610.02592} {\path{arXiv:1610.02592}}.

\bibitem{mcmullen2011evolution}
C.~T. McMullen.
\newblock The evolution of geometric structures on 3-manifolds.
\newblock {\em Bull. Amer. Math. Soc. (N.S.)}, 48(2):259--274, 2011.
\newblock \href {https://dx.doi.org/10.1090/S0273-0979-2011-01329-5}
  {\path{doi:10.1090/S0273-0979-2011-01329-5}}, \href
  {https://mathscinet.ams.org/mathscinet-getitem?mr=2774092}
  {\path{MR:2774092}}, \href {https://zbmath.org/?q=an:1214.57017}
  {\path{Zbl:1214.57017}}.

\bibitem{moise1952affine}
E.~E. Moise.
\newblock Affine structures in {$3$}-manifolds. {V}. {T}he triangulation
  theorem and {H}auptvermutung.
\newblock {\em Ann. Math. (2)}, 56:96--114, 1952.
\newblock \href {https://dx.doi.org/10.2307/1969769}
  {\path{doi:10.2307/1969769}}, \href
  {https://mathscinet.ams.org/mathscinet-getitem?mr=0048805}
  {\path{MR:0048805}}, \href {https://zbmath.org/?q=an:0048.17102}
  {\path{Zbl:0048.17102}}.

\bibitem{mostow1968quasi}
G.~D. Mostow.
\newblock Quasi-conformal mappings in {$n$}-space and the rigidity of
  hyperbolic space forms.
\newblock {\em Publ. Math. IH\'{E}S}, 34:53--104, 1968.
\newblock URL: \url{http://www.numdam.org/item?id=PMIHES_1968__34__53_0}, \href
  {https://mathscinet.ams.org/mathscinet-getitem?mr=236383} {\path{MR:236383}},
  \href {https://zbmath.org/?q=an:0189.09402} {\path{Zbl:0189.09402}}.

\bibitem{perelman2002entropy}
G.~Perelman.
\newblock The entropy formula for the {R}icci flow and its geometric
  applications, 2002.
\newblock 39 pages.
\newblock \href {https://arxiv.org/abs/math/0211159}
  {\path{arXiv:math/0211159}}.

\bibitem{perelman2003ricci}
G.~Perelman.
\newblock Ricci flow with surgery on three-manifolds, 2003.
\newblock 22 pages.
\newblock \href {https://arxiv.org/abs/math/0303109}
  {\path{arXiv:math/0303109}}.

\bibitem{porti2008geometrization}
J.~Porti.
\newblock Geometrization of three manifolds and {P}erelman's proof.
\newblock {\em Rev. R. Acad. Cienc. Exactas F\'{i}s. Nat. Ser. A Math. RACSAM},
  102(1):101--125, 2008.
\newblock \href {https://dx.doi.org/10.1007/BF03191814}
  {\path{doi:10.1007/BF03191814}}, \href
  {https://mathscinet.ams.org/mathscinet-getitem?mr=2416241}
  {\path{MR:2416241}}, \href {https://zbmath.org/?q=an:1170.57016}
  {\path{Zbl:1170.57016}}.

\bibitem{purcell2020book}
J.~S. Purcell.
\newblock Hyperbolic {K}not {T}heory (preprint).
\newblock February 2020.
\newblock URL: \url{http://users.monash.edu/~jpurcell/book/HypKnotTheory.pdf},
  \href {https://arxiv.org/abs/2002.12652} {\path{arXiv:2002.12652}}.

\bibitem{purcell2017independence}
J.~S. Purcell and A.~Zupan.
\newblock Independence of volume and genus {$g$} bridge numbers.
\newblock {\em Proc. Amer. Math. Soc.}, 145(4):1805--1818, 2017.
\newblock \href {https://dx.doi.org/10.1090/proc/13327}
  {\path{doi:10.1090/proc/13327}}, \href
  {https://mathscinet.ams.org/mathscinet-getitem?mr=3601570}
  {\path{MR:3601570}}, \href {https://zbmath.org/?q=an:1364.57009}
  {\path{Zbl:1364.57009}}.

\bibitem{robertson1983graph}
N.~Robertson and P.~D. Seymour.
\newblock Graph minors. {I.} {E}xcluding a forest.
\newblock {\em J. Comb. Theory, Ser. {B}}, 35(1):39--61, 1983.
\newblock \href {https://dx.doi.org/10.1016/0095-8956(83)90079-5}
  {\path{doi:10.1016/0095-8956(83)90079-5}}, \href
  {https://mathscinet.ams.org/mathscinet-getitem?mr=723569} {\path{MR:723569}},
  \href {https://zbmath.org/?q=an:0521.05062} {\path{Zbl:0521.05062}}.

\bibitem{robertson1986graph}
N.~Robertson and P.~D. Seymour.
\newblock Graph minors. {II.} {A}lgorithmic aspects of tree-width.
\newblock {\em J. Algorithms}, 7(3):309--322, 1986.
\newblock \href {https://dx.doi.org/10.1016/0196-6774(86)90023-4}
  {\path{doi:10.1016/0196-6774(86)90023-4}}, \href
  {https://mathscinet.ams.org/mathscinet-getitem?mr=855559} {\path{MR:855559}},
  \href {https://zbmath.org/?q=an:0611.05017} {\path{Zbl:0611.05017}}.

\bibitem{sakuma1995examples}
M.~Sakuma and J.~Weeks.
\newblock Examples of canonical decompositions of hyperbolic link complements.
\newblock {\em Japan. J. Math. (N.S.)}, 21(2):393--439, 1995.
\newblock \href {https://dx.doi.org/10.4099/math1924.21.393}
  {\path{doi:10.4099/math1924.21.393}}, \href
  {https://mathscinet.ams.org/mathscinet-getitem?mr=1364387}
  {\path{MR:1364387}}, \href {https://zbmath.org/?q=an:0858.57021}
  {\path{Zbl:0858.57021}}.

\bibitem{saveliev2012lectures}
N.~Saveliev.
\newblock {\em Lectures on the Topology of 3-Manifolds: An Introduction to the
  Casson Invariant}.
\newblock De Gruyter Textbook. Walter de Gruyter \& Co., Berlin, 2nd edition,
  2012.
\newblock \href {https://dx.doi.org/10.1515/9783110250367}
  {\path{doi:10.1515/9783110250367}}, \href
  {https://mathscinet.ams.org/mathscinet-getitem?mr=2893651}
  {\path{MR:2893651}}, \href {https://zbmath.org/?q=an:1246.57003}
  {\path{Zbl:1246.57003}}.

\bibitem{scharlemann2002heegaard}
M.~Scharlemann.
\newblock Heegaard splittings of compact 3-manifolds.
\newblock In {\em Handbook of Geometric Topology}, pages 921--953.
  North-Holland, Amsterdam, 2001.
\newblock \href {https://dx.doi.org/10.1016/B978-044482432-5/50019-6}
  {\path{doi:10.1016/B978-044482432-5/50019-6}}, \href
  {https://mathscinet.ams.org/mathscinet-getitem?mr=1886684}
  {\path{MR:1886684}}, \href {https://zbmath.org/?q=an:0985.57005}
  {\path{Zbl:0985.57005}}.

\bibitem{scharlemann2016lecture}
M.~Scharlemann, J.~Schultens, and T.~Saito.
\newblock {\em Lecture Notes on Generalized {H}eegaard Splittings}.
\newblock World Scientific Publishing Co. Pte. Ltd., Hackensack, NJ, 2016.
\newblock \href {https://dx.doi.org/10.1142/10019} {\path{doi:10.1142/10019}},
  \href {https://mathscinet.ams.org/mathscinet-getitem?mr=3585907}
  {\path{MR:3585907}}, \href {https://zbmath.org/?q=an:1356.57004}
  {\path{Zbl:1356.57004}}.

\bibitem{scharlemann1992thin}
M.~Scharlemann and A.~Thompson.
\newblock Thin position for {$3$}-manifolds.
\newblock In {\em Geometric topology ({H}aifa, 1992)}, volume 164 of {\em
  Contemp. Math.}, pages 231--238. Amer. Math. Soc., Providence, RI, 1994.
\newblock \href {https://dx.doi.org/10.1090/conm/164/01596}
  {\path{doi:10.1090/conm/164/01596}}, \href
  {https://mathscinet.ams.org/mathscinet-getitem?mr=1282766}
  {\path{MR:1282766}}, \href {https://zbmath.org/?q=an:0818.57013}
  {\path{Zbl:0818.57013}}.

\bibitem{schultens1993classification}
J.~Schultens.
\newblock The classification of {H}eegaard splittings for (compact orientable
  surface){$\,\times\, S^1$}.
\newblock {\em Proc. London Math. Soc. (3)}, 67(2):425--448, 1993.
\newblock \href {https://dx.doi.org/10.1112/plms/s3-67.2.425}
  {\path{doi:10.1112/plms/s3-67.2.425}}, \href
  {https://mathscinet.ams.org/mathscinet-getitem?mr=1226608}
  {\path{MR:1226608}}, \href {https://zbmath.org/?q=an:0789.57012}
  {\path{Zbl:0789.57012}}.

\bibitem{schultens2014introduction}
J.~Schultens.
\newblock {\em Introduction to 3-Manifolds}, volume 151 of {\em Grad. Stud.
  Math.}
\newblock Amer. Math. Soc., Providence, RI, 2014.
\newblock \href {https://dx.doi.org/10.1090/gsm/151}
  {\path{doi:10.1090/gsm/151}}, \href
  {https://mathscinet.ams.org/mathscinet-getitem?mr=3203728}
  {\path{MR:3203728}}, \href {https://zbmath.org/?q=an:1295.57001}
  {\path{Zbl:1295.57001}}.

\bibitem{shalen2007hyperbolic}
P.~B. Shalen.
\newblock Hyperbolic volume, {H}eegaard genus and ranks of groups.
\newblock In {\em Workshop on {H}eegaard {S}plittings}, volume~12 of {\em Geom.
  Topol. Monogr.}, pages 335--349. Geom. Topol. Publ., Coventry, 2007.
\newblock \href {https://dx.doi.org/10.2140/gtm.2007.12.335}
  {\path{doi:10.2140/gtm.2007.12.335}}, \href
  {https://mathscinet.ams.org/mathscinet-getitem?mr=2408254}
  {\path{MR:2408254}}, \href {https://zbmath.org/?q=an:1140.57009}
  {\path{Zbl:1140.57009}}.

\bibitem{takahashi1994minimal}
A.~Takahashi, S.~Ueno, and Y.~Kajitani.
\newblock Minimal acyclic forbidden minors for the family of graphs with
  bounded path-width.
\newblock {\em Discrete Math.}, 127(1--3):293--304, 1994.
\newblock \href {https://dx.doi.org/10.1016/0012-365X(94)90092-2}
  {\path{doi:10.1016/0012-365X(94)90092-2}}, \href
  {https://mathscinet.ams.org/mathscinet-getitem?mr=1273610}
  {\path{MR:1273610}}, \href {https://zbmath.org/?q=an:0795.05123}
  {\path{Zbl:0795.05123}}.

\bibitem{thurston1982three}
W.~P. Thurston.
\newblock Three-dimensional manifolds, {K}leinian groups and hyperbolic
  geometry.
\newblock {\em Bull. Amer. Math. Soc. (N.S.)}, 6(3):357--381, 1982.
\newblock \href {https://dx.doi.org/10.1090/S0273-0979-1982-15003-0}
  {\path{doi:10.1090/S0273-0979-1982-15003-0}}, \href
  {https://mathscinet.ams.org/mathscinet-getitem?mr=648524} {\path{MR:648524}},
  \href {https://zbmath.org/?q=an:0496.57005} {\path{Zbl:0496.57005}}.

\bibitem{thurston2014three}
W.~P. Thurston.
\newblock {\em Three-Dimensional Geometry and Topology. {V}ol. 1}, volume~35 of
  {\em Princeton Math. Ser.}
\newblock Princeton Univ. Press, Princeton, NJ, 1997.
\newblock Edited by S. Levy.
\newblock \href {https://dx.doi.org/10.1515/9781400865321}
  {\path{doi:10.1515/9781400865321}}, \href
  {https://mathscinet.ams.org/mathscinet-getitem?mr=1435975}
  {\path{MR:1435975}}, \href {https://zbmath.org/?q=an:0873.57001}
  {\path{Zbl:0873.57001}}.

\bibitem{thurston2002geometry}
W.~P. Thurston.
\newblock The geometry and topology of three-manifolds.
\newblock Electronic version 1.1, March, 2002.
\newblock URL: \url{http://library.msri.org/books/gt3m}.

\end{thebibliography}

\clearpage

\appendix

\section{The primal and dual construction of compression bodies}
\label{app:compbody}

\begin{figure}[!ht]
\centerline{\includegraphics[scale=\MyFigScale]{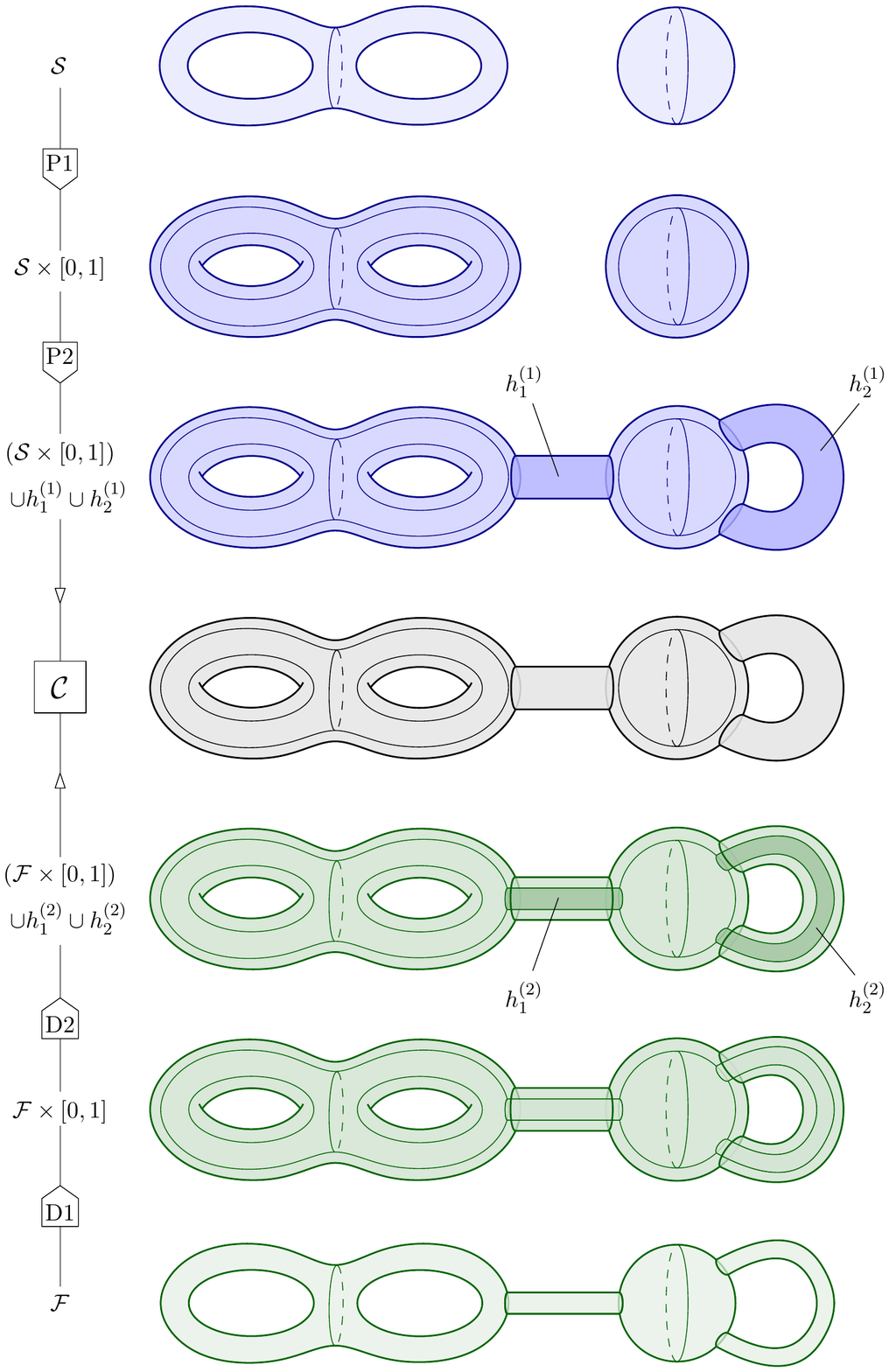}}
\caption{The primal and dual ways of constructing a compression body $\compbody$, cf.\ Section \ref{ssec:bodies}.}
\label{fig:compbody}
\end{figure}

\clearpage

\section{Heegaard splittings of 3-manifolds with boundary}
\label{app:heegaard}

\begin{example}[Heegaard splittings from triangulations, II -- based on  {\cite[Theorem 2.1.11]{scharlemann2016lecture}}]
\label{ex:heegaard_boundary}
Let $\tri$ be a triangulation of $\manifold$ with partition $\partial_1\manifold \cup \partial_2\manifold$ of its boundary components. Suppose that no simplex in $\tri$ is incident to more than one component of $\partial\manifold$.\footnote{This can be achieved, e.g., by passing to the first barycentric subdivision of $\tri$ if necessary.} Take the first barycentric subdivision $\sd_1(\tri)$ of $\tri$. Recall that $\tri^{(1)}$ and $\Gamma(\tri)$ denote the 1-skeleton and the dual graph of $\tri$. Their first barycentric subdivisons $\tri^{(1)}_{\sd}$ and $\Gamma(\tri)_{\sd}$ are both naturally contained in $\sd_1(\tri)$. Consider the subcomplex $\nbh{\partial_2\manifold} \subset \sd_1(\tri)$ consisting of all simplices incident to $\partial_2\manifold$.  We define two further subcomplexes of $\sd_1(\tri)$, namely
\begin{itemize}
	\item $\Gamma_1 = \partial_1\manifold \cup \{\text{vertices and edges of $\tri^{(1)}_{\sd}$ {\em not} incident to $\partial_2\manifold$}\}$, and
	\item $\Gamma_2 = \nbh{\partial_2\manifold} \cup \Gamma(\tri)_{\sd}.$
\end{itemize}
Now pass to the second barycentric subdivision $\sd_2(\tri)$ and let $(\Gamma_i)_{\sd}$ denote the image of $\Gamma_i$ under this operation ($i=1,2$). Let $\eta(\Gamma_i)$ be the ``thickening'' of $\Gamma_i$, i.e., the subcomplex of $\sd_2(\tri)$ formed by all simplices incident to $(\Gamma_i)_{\sd}$.
One can readily verify that $\eta(\Gamma_1)$ and $\eta(\Gamma_2)$ are compression bodies whose union is $\manifold$, their upper boundaries satisfy $\partial_+\eta(\Gamma_1) = \partial_+\eta(\Gamma_2) = \eta(\Gamma_1) \cap \eta(\Gamma_2)$, and for their lower boundaries $\partial_-\eta(\Gamma_1) = \partial_1\manifold$ and $\partial_-\eta(\Gamma_2) = \partial_2\manifold$.
Hence $\eta(\Gamma_1)$ and $\eta(\Gamma_2)$ form a Heegaard splitting of $\manifold$ compatible with the given partition of its boundary components. See Figure \ref{fig:heegaard_boundary} for an illustration via ``quadrangulations.''
\end{example}

\vfill

\begin{figure}[ht]
	\noindent
	\begin{minipage}{0.325\textwidth}
		\centering
		\includegraphics[scale=\MyFigScale]{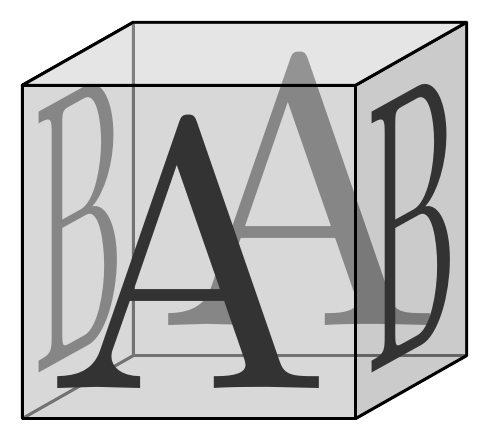}
		
		{\small (i) $\manifold = \torus^2 \times [0,1]$ \phantom{lorem ipsum dolor sit amet}}
	\end{minipage}~
	\begin{minipage}{0.325\textwidth}
		\centering
		\includegraphics[scale=\MyFigScale]{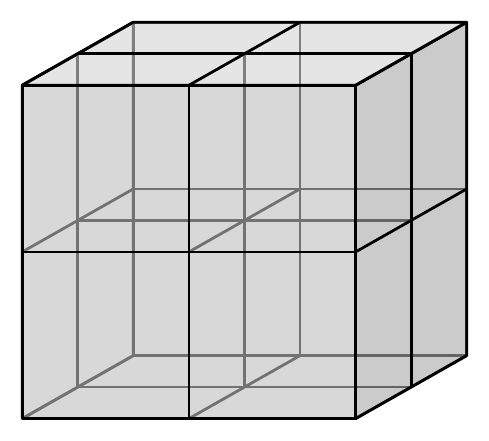}

		{\small (ii) A quadrangulation $\mathcal{Q}$ of $\manifold$ with eight cubes}
	\end{minipage}~
	\begin{minipage}{0.325\textwidth}
		\centering
		\includegraphics[scale=\MyFigScale]{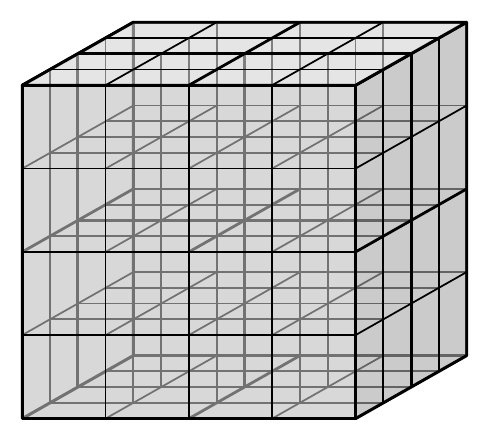}

		{\small (iii) The first barycentric subdivision $\operatorname{sd}_1(\mathcal{Q})$ of $\mathcal{Q}$}
	\end{minipage}

	\vspace{18pt}

	\noindent
	\begin{minipage}{0.325\textwidth}
		\centering
		\includegraphics[scale=\MyFigScale]{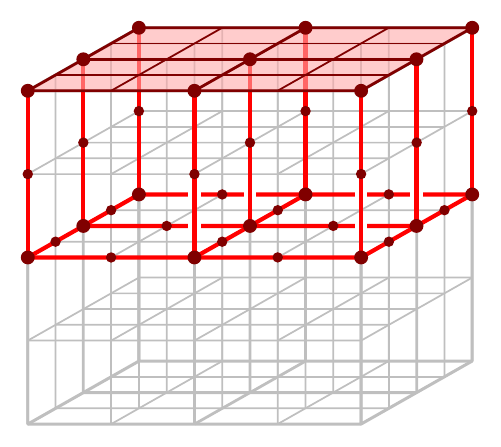}

		{\small (iv) ${\color{BrickRed}\boldsymbol{\Gamma_1}}$}
	\end{minipage}~
	\begin{minipage}{0.325\textwidth}

	\vspace{12pt}
	
	\begin{flushleft}
		\small{${\color{BrickRed}\boldsymbol{\Gamma_1}} = \partial_1\manifold \cup \{$vertices \& edges of $\tri^{(1)}_{\sd}$ avoiding $\nbh{\partial_2\manifold}\}$}
		
	\vspace{30pt}
	
	\hfill\begin{minipage}{\dimexpr\textwidth-1cm}
		\small{${\color{Blue}\boldsymbol{\Gamma_2}} = \nbh{\partial_2\manifold} \cup \Gamma(\tri)_{\sd}$}	
	\end{minipage}
	\end{flushleft}
	
	{\small\phantom{(v) ${\color{Blue}\boldsymbol{\Gamma_2}}$}}
	\end{minipage}~
	\begin{minipage}{0.325\textwidth}
		\centering
		\includegraphics[scale=\MyFigScale]{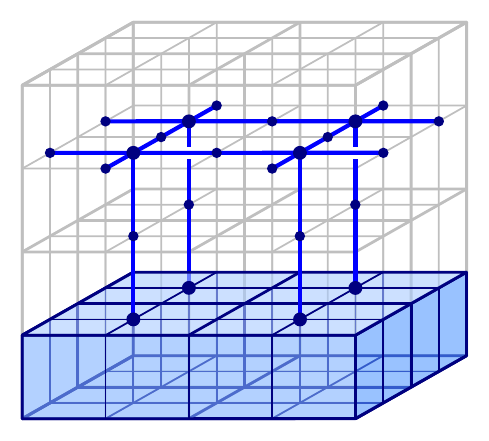}

		{\small (v) ${\color{Blue}\boldsymbol{\Gamma_2}}$}
	\end{minipage}

	\vspace{6pt}
	\caption{Building a Heegaard splitting of the thickened torus $\torus^2 \times [0,1]$ from a quadrangulation.}
	\label{fig:heegaard_boundary}
\end{figure}

\vfill

\end{document}